\documentclass{article}
\usepackage{amsthm,amsmath,amssymb,amsfonts,graphicx,enumitem}
  \newtheorem{thm}{Theorem}[section]
  \newtheorem{lem}[thm]{Lemma}
  \newtheorem{prop}[thm]{Proposition}
  \newtheorem{cor}[thm]{Corollary}
  \theoremstyle{definition}
  \newtheorem{defn}[thm]{Definition}
  \newtheorem{exm}[thm]{Example}
  \newtheorem{rmk}[thm]{Remark}
  
 \newcommand\ra{\rightarrow}

 \newcommand\mI{\mathcal{I}}

 \newcommand\s{\subseteq}

 \newcommand\Hom{\mathrm{Hom}}

 \newcommand\e{\textsf{e}}
  \newcommand\lam{\lambda}

 \numberwithin{equation}{section}
\def\iff{if and only if }
\def\im{ \mathrm{Im} }

\def\iff{if and only if }
\begin{document}
\title{{\bf Morphisms on $EMV$-algebras and Their Applications} }

\author{ Anatolij Dvure\v{c}enskij$^{^{1,2}}$, Omid Zahiri$^{^{3}}$ \\ 
{\small\em $^1$Mathematical Institute,  Slovak Academy of Sciences, \v Stef\'anikova 49, SK-814 73 Bratislava, Slovakia} \\
{\small\em $^2$Depart. Algebra  Geom.,  Palack\'{y} Univer., 17. listopadu 12, CZ-771 46 Olomouc, Czech Republic} \\
{\small\em  $^{3}$University of Applied Science and Technology, Tehran, Iran}\\
{\small\tt  dvurecen@mat.savba.sk\quad   zahiri@protonmail.com} }
\date{}
\maketitle
\begin{abstract}
For a new class of algebras, called $EMV$-algebras, every idempotent element $a$ determines an $MV$-algebra which is important for the structure of the $EMV$-algebra. Therefore, instead of standard homomorphisms of $EMV$-algebras, we introduce $EMV$-morphisms as a family of $MV$-homomorphisms from $MV$-algebras $[0,a]$ into other ones. $EMV$-morphisms enable us to study categories of $EMV$-algebras where objects are $EMV$-algebras and morphisms are special classes of $EMV$-morphisms. The category is closed under product. In addition, we define free $EMV$-algebras on a set $X$ with respect to $EMV$-morphisms. If $X$ is finite, then the free $MV$-algebra on $X$ is a free $EMV$-algebras. For an infinite set $X$, the same is true introducing a so-called weakly free $EMV$-algebra.
\end{abstract}



{\small {\it Acknowledgement:} AD is thankful for the support by the Slovak Research and Development Agency under the contract No. APVV-16-0073
and by grants VEGA No. 2/0069/16 SAV and GA\v{C}R 15-15286S  }

\section{ Introduction }
Boolean algebras are well-known structures studied over many decades. They describe an algebraic semantics for two-valued logic. In Thirties, there appeared Boolean rings, or equivalently, generalized Boolean
algebras, which have almost Boolean features, but a top element is not
assumed. For such structures, Stone, see e.g. \cite[Thm 6.6]{LuZa}, developed a representation of Boolean rings by rings of subsets, and also some logical models with such an incomplete information were established, see \cite{Sto1,Sto2}.

Recently in \cite{Dvz}, a \L ukasiewicz type algebraic structure with incomplete total information was developed, i.e. we found an algebraic semantics very similar to $MV$-algebras with incomplete information, which however in a local sense is complete: Conjunctions and disjunctions exist, negation only locally, i.e. negation of $a$ in $b$ exists only if $a\le b$ but total negation of the event $a$ is not assumed.
For such ideas we have introduced in \cite{Dvz} $EMV$-algebras which are locally close to $MV$-algebras, however, the top element is not assumed. The basic representation theorem says, \cite[Thm 5.21]{Dvz}, that even in such a case, we can find an $MV$-algebra where the original algebra can be embedded as its maximal ideal, i.e. an incomplete information hidden in an $EMV$-algebra is sufficient to find a \L ukasiewicz logical system where a top element exists and where all original statements are valid.

Of course, every $MV$-algebra is an $EMV$-algebra ($EMV$-algebras stand for extended $MV$-algebras), and $EMV$-algebras generalize Chang's $MV$-algebras, \cite{Cha}. Nowadays $MV$-algebras have many important applications in different areas of mathematics and logic. Therefore, $MV$-algebras have many different generalizations, like $BL$-algebras, pseudo $MV$-algebras, \cite{georgescu, Dvu2}, $GMV$-algebras in the realm of residuated lattices, \cite{Tsinakis}, etc. In the last period $MV$-algebras are studied also in frames of involutive semirings, see \cite{DiRu}. The presented $EMV$-algebras are another kind of generalizations of $MV$-algebras inspired by Boolean rings.

In \cite{Dvz1} we have formulated and proved a Loomis--Sikorski-type theorem for $\sigma$-complete $EMV$-algebras showing that every $\sigma$-complete $EMV$-algebra is a $\sigma$-homomorphic image of an $EMV$-tribe of fuzzy sets, where all $EMV$-operations are defined by points. To show this, we have introduced the hull-kernel topology of the set of maximal ideals of an $EMV$-algebra and the weak topology of state-morphisms which are $EMV$-homomorphisms from the $EMV$-algebra into the $MV$-algebra of the real interval $[0,1]$, or equivalently, a variant of extremal probability measures.

In this paper we will propose a new definition for morphisms between $EMV$-algebras, called $EMV$-morphisms, which is more relevant to this structure. An $EMV$-morphism from $M_1$ to $M_2$ is a family of $MV$-homomorphisms  with special properties, Section 3.
We define similarity, ``$\approx$'', and composition of morphisms, ``$\circ$'', of two $EMV$-morphisms
which are presented in Definition \ref{hhhh} and Proposition \ref{9.2}.  We prove that ``$\approx$'' and ``$\circ$'' are compatible.
In Section 4, we introduce three categories of $EMV$-algebras.
We find some properties of $EMV$-morphisms and show that the category of $EMV$-algebras with this new morphisms is closed under products, Section 5.
Finally in Section 6, we study free $EMV$-algebras on a set $X$ with respect to our $EMV$-morphisms. We show that the free $MV$-algebra on a finite set $X$ is also free in this class. For an infinite set $X$, we
introduce a new object, called a weakly free $EMV$-algebra, which is very similar to the free object, and we show that the free $MV$-algebra on any set $X$ is the weakly free $EMV$-algebra on $X$.

\section{ Preliminaries}

We remind that an $MV$-{\it algebra} is an algebra $(M;\oplus,',0,1)$ (henceforth write simply $M=(M;\oplus,',0,1)$) of type $(2,1,0,0)$, where $(M;\oplus,0)$ is a
commutative monoid with the neutral element $0$ and for all $x,y\in M$, we have:
\vspace{1mm}
\begin{enumerate}[nolistsep]
	\item[(i)] $x''=x$;
	\item[(ii)] $x\oplus 1=1$;
	\item[(iii)] $x\oplus (x\oplus y')'=y\oplus (y\oplus x')'$.
\end{enumerate}
\vspace{2mm}
It is well-known that every $MV$-algebra is a distributive lattice. For more info about $MV$-algebras, see \cite{mundici 1}.

Let $(M;\oplus,0)$ be a commutative monoid. An element $a \in M$ is said to be {\it idempotent} if $a\oplus a = a$. We denote by $\mathcal I(M)$ the set of idempotents of $M$. Then (i) $0 \in \mathcal I(M)$, and $a,b \in \mathcal I(M)$ entail $a\oplus b \in \mathcal I(M)$. A non-empty subset $S$ of idempotents of $M$ is {\it full} if, given $b \in \mathcal I(M)$, there is $a \in S$ such that $b\le a$. A commutative monoid $(M;\oplus,0)$ endowed with a partial order $\le$ is {\it ordered} if $x\le y$ implies $x\odot z\le y\odot z$.

\begin{defn}\label{de:GMV}\cite{Dvz}
	An algebra $(M;\vee,\wedge,\oplus,0)$ of type $(2,2,2,0)$ is called an
	{\it extended $MV$-algebra}, an {\it $EMV$-algebra} in short,  if it satisfies the following conditions:
	\vspace{1mm}
	\begin{itemize}[nolistsep]
		\item[]{\rm ($EMV1$)}\ $(M;\vee,\wedge,0)$ is a distributive lattice with the least element $0$;
		\item[]{\rm ($EMV2$)}\  $(M;\oplus,0)$ is a commutative ordered monoid with neutral element $0$;
		\item[]{\rm ($EMV3$)}\ for each $b\in \mI(M)$, the element
		$$
		\lambda_{b}(x)=\min\{z\in[0,b]\mid x\oplus z=b\}
		$$
		exists in $M$ for all $x\in [0,b]$, and the algebra $([0,b];\oplus,\lambda_{b},0,b)$ is an $MV$-algebra;
		\item[]{\rm ($EMV4$)}\ for each $x\in M$, there is $a\in \mI(M)$ such that
		$x\leq a$.
	\end{itemize}
\end{defn}
An $EMV$-algebra $M$ is called {\em proper} if $M$ is not an $MV$-algebra or, equivalently, $M$ has not a top element.

We note that every $MV$-algebra can be viewed as an $EMV$-algebra, every Boolean ring (= generalized Boolean algebra) is also an example of $EMV$-algebras. The basic properties of $EMV$-algebras were presented in \cite{Dvz}, where there was proved, in particular, that the class of $EMV$-algebras forms a variety.

Let $(M;\vee,\wedge,\oplus,0)$ be an $EMV$-algebra. Then for all $a\in\mI(M)$, we have a well-known binary operation $\odot_a$ on the $MV$-algebra $([0,a];\oplus,\lam_a,0,a)$ given by
$x\odot_{_a} y=\lam_a(\lam_a(x)\oplus \lam_a(y))$ for all $x,y \in [0,a]$. It can be extended to $\odot$ defined on the whole $M$, see Lemma \ref{le:x<y} below.

Let $(M;\vee,\wedge,\oplus,0)$ be an $EMV$-algebra. Its reduct
$(M;\vee,\wedge,0)$ is a distributive lattice with a bottom element $0$. The lattice structure of $M$  yields a partial order relation on $M$, denoted by $\leq$, that is $x\leq y$ iff $x\vee y=y$ iff $x\wedge y=x$. If $a$ is a fixed idempotent element of $M$,
$([0,a];\oplus,\lambda_a,0,a)$ is an $MV$-algebra. Recall that, in each $MV$-algebra $(A;\oplus,',0,1)$
there is a partial order relation $\preccurlyeq$ (induced by $\oplus,',1$) defined by $x\preccurlyeq  y$ iff $x'\oplus y=1$. So, the partial order on
the $MV$-algebra $([0,a];\oplus,\lambda_a,0,a)$ is defined by $x\preccurlyeq y$ iff $\lambda_a(x)\oplus y=a$. In the sequel, we
show that, for each $x,y\in [0,a]$, we have
\[x\leq y\Leftrightarrow x\preccurlyeq y.\]
First, we assume that $x\preccurlyeq y$. Then $\lambda_a(x)\oplus y=a$. Set $z:=\lambda_a(x)$. Then
$z\oplus y=a$ entails $\lambda_a(z)\leq y$ (by definition) and so $\lambda_a(\lambda_a(x))\leq y$.  Thus $x\leq y$.
Note that $\lambda_a(x)=\min\{u\in [0,a]\mid x\oplus u=a\}$ where $\min$ is related to $\leq$.

Conversely, let $x\leq y$. We show that $x\preccurlyeq y$ or equivalently, $\lambda_a(x)\oplus y=a$.
By definition of $\lambda_a(x)$, we have $x\oplus \lambda_a(x)=a$. Since $x\leq y$ and $([0,a];\oplus ,0)$ is a partially ordered monoid
(see $EMV2$), then $a=x\oplus\lambda_a(x)\leq y\oplus\lambda_a(x)\leq a$ and so $y\oplus \lambda_a(x)=a$, which
implies that $x\preccurlyeq y$.

Finally, we can easily prove that, for each $z,w\in [0,a]$,
the supremum and infimum of the set $\{z,w\}$ in $[0,a]$ coincide with ones of $\{w,z\}$ taken in $M$.

Moreover, it is possible to show that given $x,y \in M$, $x\le y$ iff there is an element $z\in M$ such that $x\oplus z = y$. Indeed, there is an idempotent $a\in M$ such that $x,y \le a$. Then in the $MV$-algebra $[0,a]$, we have $x\vee y= x\oplus (y\odot_a \lambda_a(x))$, that is, $M$ is a naturally ordered monoid.

\begin{defn}\label{3.4}\cite{Dvz}
	(i)  Let $(M;\vee,\wedge,\oplus,0)$ be an $EMV$-algebra. A subset $A\s M$ is called an $EMV$-{\it subalgebra} of $M$ if
\begin{itemize}[nolistsep]
\item[{\rm (1)}]  $A$ is closed
	under $\vee$, $\wedge$, $\oplus$ and $0$;
\item[{\rm (2)}] 	for all $x\in A$,  there is $b\in A\cap \mI(M)$ such that $x\leq b$;

 \item[{\rm (3)}] for each $b\in \mI(M)\cap A$ the set $[0,b]_A:=[0,b]\cap A$ is a subalgebra
	of the $MV$-algebra $([0,b];\oplus,\lam_b,0,b)$.
\end{itemize}
Clearly, the last condition is equivalent to the following condition:
	$$\forall\, b\in A\cap \mI(M),\quad \forall\, x\in [0,b]_A,\ \ \min\{z\in [0,b]_A\mid x\oplus z=b\}=\min\{z\in [0,b]\mid x\oplus z=b\},
	$$
	and this also means $\lambda_{b}(x)$ is defined in $[0,b]_A$ for each $b\in \mathcal I (M)\cap A$ and for each $x\in [a,b]_A$.
	
	(ii) Let $(M_1;\vee,\wedge,\oplus,0)$ and $(M_2;\vee,\wedge,\oplus,0)$ be $EMV$-algebras. A map $f:M_1\ra M_2$ is called an $EMV$-{\it homomorphism}
	if $f$ preserves the operations $\vee$, $\wedge$, $\oplus$ and $0$, and for each $b\in\mI(M_1)$ and for each $x\in [0,b]$, $f(\lam_b(x))= \lam_{f(b)}(f(x))$.
	An $EMV$-homomorphism $f:M_1\ra M_2$ is called {\em strong} if $\{f(a)\mid a\in\mI(M_1)\}$ is a full subset of $M_2$ (that is, for each $b\in \mI(M_2)$, there exists $a\in \mathcal I(M_1)$ such that $b\leq f(a)$). Clearly, every identity $Id_M$ on an $EMV$-algebra is a strong $EMV$-homomorphism.

\end{defn}

\begin{prop}\label{3.5}{\rm \cite[Prop 3.9]{Dvz}}
	Let $(M;\vee,\wedge,\oplus,0)$ be an $EMV$-algebra, $a,b\in \mI(M)$ such that $a\leq b$.
	Then  for each $x\in [0,a]$, we have
	\vspace{1mm}
	\begin{itemize}[nolistsep]
		\item[{\rm (i)}]  $\lam_a(x)=\lam_b(x)\wedge a$;
		\item[{\rm (ii)}]   $\lam_b(x)=\lam_a(x)\oplus \lam_b(a)$;
		\item[{\rm (iii)}] $\lam_b(a)$ is an idempotent, and $\lam_a(a)=0$.
	\end{itemize}
\end{prop}

\begin{defn}\label{3.8}\cite{Dvz}
	Let $(M;\vee,\wedge,\oplus,0)$ be an $EMV$-algebra. An equivalence relation $\theta$ on $M$ is called a {\it congruence relation} or simply a \emph{congruence} if
	it satisfies the following conditions:
	\vspace{1mm}
	\begin{itemize}[nolistsep]
		\item[(i)]  $\theta$ is compatible with $\vee$, $\wedge$ and $\oplus$;
		
		\item[(ii)] for all $b\in \mI(M)$, $\theta\cap ([0,b]\times [0,b])$ is a congruence relation on the $MV$-algebra $([0,b];\oplus,\lam_b,0,b)$.
	\end{itemize}
	We denote by $\mathrm{Con}(M)$ the set of all congruences on $M$.
\end{defn}

\begin{prop}\label{3.9} {\rm \cite[Prop 3.13]{Dvz}}
	An equivalence relation $\theta$ on an $EMV$-algebra $(M;\vee,\wedge,\oplus,0)$ is a congruence if it is compatible with $\vee$, $\wedge$ and $\oplus$, and
	for all $(x,y)\in\theta$, there exists $b\in\mI(M)$ such that $x,y\leq b$ and $(\lam_b(x),\lam_b(y))\in\theta$.
\end{prop}

Let $\theta$ be a congruence relation on an $EMV$-algebra $(M;\vee, \wedge,\oplus,0)$ and $M/\theta=\{[x]\mid x\in M\}$ (we usually use $x/\theta$ instead of $[x]$). Consider the induced operations $\vee$, $\wedge$ and $\oplus$ on $M/\theta$ defined by
$$
[x]\vee[y]=[x\vee y],\  \  [x]\wedge [y]=[x\wedge y],\   \  [x]\oplus [y]=[x\oplus y],\quad \forall\, x,y\in M.
$$
In addition, $[\lambda_b(x)]=\lambda_{[b]}([x])$ for each $b \in \mathcal I(M)$.
Then $(M/\theta;\vee,\wedge,\oplus,0/\theta)$ is an $EMV$-algebra, and the mapping $x\mapsto x/\theta$ is an $EMV$-homomorphism from $M$ onto $M/\theta$.

\begin{lem}\label{le:x<y}{\rm \cite[Lem 5.1]{Dvz}}
Let $(M;\vee, \wedge,\oplus,0)$ be an $EMV$-algebra. For all $x,y\in M$, we define
$$
x\odot y=\lam_a(\lam_a(x)\oplus \lam_a(y)),
$$
where $a\in\mI(M)$ and $x,y\leq a$. Then $\odot:M\times M\ra M$ is an order preserving, associative well-defined binary operation on $M$ which does not depend on $a\in \mI(M)$.
In addition, if $x,y \in M$, $x\le y$, then $y \odot \lambda_a(x)=y\odot \lambda_b(x)$
for all idempotents $a,b$ of $M$ with $x,y\le a,b$.
\end{lem}

For any integer $n\ge 1$ and any $x$ of an $EMV$-algebra $M$, we can define $ x^1 =x$, $x^n=x^{n-1}\odot x$, $n\ge 2$,
and if $M$ has a top element $1$, we define also $x^0=1$.

We note that a non-void subset $I$ of an $EMV$-algebra $M$ is an {\it ideal} if (i) if $a\le b \in I$, then $a \in I$, and (ii) $a,b \in I$ gives $a\oplus b \in I$.

We have already said that not every $EMV$-algebra $M$ possesses a top element. Anyway, in such a case,  it can be embedded into an $MV$-algebra as its maximal ideal as the following representation theorem says:

\begin{thm}\label{th:embed}{\rm \cite[Thm 5.21]{Dvz}} {\rm [Basic Representation Theorem]}
Every $EMV$-algebra $M$ is either an $MV$-algebra or $M$ can be embedded into an $MV$-algebra $N$ as a maximal ideal of $N$.
\end{thm}
\section{$EMV$-morphisms and $EMV$-algebras}

We know that an $EMV$-algebra $M$ is locally an $MV$-algebra, that is for each idempotent element $a\in M$, $[0,a]$ is an $MV$-algebra.
So, it is natural to introduce a new definition for the concept of a morphism between $EMV$-algebras applying this property; we call it an $EMV$-morphism.  If $M_1$ and $M_2$ are $EMV$-algebras, a new homomorphism needs not be necessary a map from $M_1$ to $M_2$. It is a family of $MV$-homomorphisms $\{f_i\mid i\in I\}$ with special properties. This definition is a generalization of that one introduced in \cite{Dvz}, see Definition \ref{3.4}. The main purpose of the section is to introduce and study the basic properties of $EMV$-morphisms. We introduce an equivalence  $\approx$ between $EMV$-morphisms from an $EMV$-algebra $M_1$ into another one $M_2$ and a composition of two $EMV$-morphisms is established. Standard $EMV$-morphisms will play an important role.

\begin{defn}\label{9.1}
Let $M_1$ and $M_2$ be $EMV$-algebras. An {\em $EMV$-morphism} $f:M_1\ra M_2$ is a family $f=\{f_i \mid i \in I\}$ of $MV$-homomorphisms $f_i:[0,a_i]\to [0,b_i]$ for each $i \in I$, where $\{a_i\mid i \in I\}$ and $\{b_i\mid i \in I\}$ are  non-empty sets of idempotents of $M_1$ and $M_2$, respectively,
	 such that
	\begin{itemize}[nolistsep]
		\item[(i)] $\textsf{e}(f):=\{a_i\mid i\in I\}$ is a full subset of $\mI(M_1)$;
		\item[(ii)] $\{b_i\mid i\in I\}$ is a full subset of $M_2$;
		\item[(iii)] if $i,j\in I$ and $f_i(a_i)\leq f_j(a_j)$, then $f_i(x)=f_j(x) \wedge f_i(a_i)$ for all $x\in [0,a_i\wedge a_j]$;
		\item[(iv)] for each $i,j\in I$, there exists $t\in I$ such that $a_i,a_j\leq a_t$ and $f_i(a_i),f_j(a_j)\leq f_t(a_t)$.
	\end{itemize}
Then $b_i=f_i(a_i)$ for each $i \in I$.
We use $\im(f)$ to denote $\bigcup_{i\in I}\im(f_i):=\bigcup_{i\in I}f_i([0,a_i])$.
The set of all $EMV$-morphisms from $M_1$ to $M_2$ is denoted by $\Hom(M_1,M_2)$.
\end{defn}
We note that if in $f=\{f_i\mid i\in I\}$ every $f_i$ is an $EMV$-homomorphism from the $MV$-algebra $[0,a_i]$ into $M_2$, then each $f_i$ is also an $MV$-homomorphism from the $MV$-algebra $[0,a_i]$ into the $MV$-algebra $[0,b_i]$, where $b_i=f_i(a_i)$. Therefore, without misunderstanding, we can assume formally that each $f_i$ is an $MV$-homomorphism from $[0,a_i]$ into $[0,f_i(a_i)]$.

The basic properties of $EMV$-morphisms are as follows.

\begin{prop}\label{pr:properties}
Let $f=\{f_i\mid i \in I\}$ with $\mathsf{e}(f)=\{a_i\mid i \in I\}$ be an $EMV$-morphism.

\begin{itemize}[nolistsep]
\item[{\rm (1)}] If $f_i(a_i)\le f_j(a_j)$, then for each $x \in [0,a_i\wedge a_j]$, $f_i(x)\le f_j(x)$.

\item[{\rm (2)}] If $f_i(a_i)= f_j(a_j)$, then for each $x \in [0,a_i\wedge a_j]$, $f_i(x)= f_j(x)$.

\item[{\rm (3)}]
For each $a_i$, there exists $a_t$ with $a_i\le a_t$ such that $f_i(a_i)\le f_t(a_t)$ and for each $x\in [0,a_i]$, $f_i(x)\le f_t(x)$.

\item[{\rm (4)}] If for $a_i$, there exists $a_t$ with $a_i\le a_t$ such that $f_i(a_i)=f_t(a_t)$, then for each $x\in [0,a_i]$, $f_i(x)= f_t(x)$.
\end{itemize}
\end{prop}

\begin{proof}
(1) Using (iii) of Definition \ref{9.1}, we have $f_i(x)=f_j(x)\wedge f_i(a_i)\le f_j(x)$, so that $f_i(x)\le f_j(x)$.

(2) Let $f_i(a_i)= f_j(a_j)$. Due to (iii), we get $f_i(x)=f_j(x)\wedge f_i(a_i)=f_j(x)\wedge f_j(a_j) = f_j(x)$.

(3) This follows from (1).

(4) It follows from (2).
\end{proof}

\begin{exm}\label{ex:MV}
If $M_1$ and $M_2$ are $EMV$-algebras with a top element, i.e. they are $MV$-algebras, situation with $EMV$-morphisms is as follows.

(1) If  $f:M_1\ra M_2$ is an $MV$-homomorphism, then $I=\{1\}$ is a full subset of $M_1$, and the singleton $\{f_1\}$ is an $EMV$-morphism.

(2) If $g=\{g_i\mid i\in I\}: M_1\ra M_2$ with $\mathsf{e}(g)=\{a_i\mid i \in I\}$ is an $EMV$-morphism of $EMV$-algebras, then clearly, there is $1=a_j\in \{a_i\mid i\in I\}$ such that $g_j(a_j)=1$ and so $g_j:M_1\ra M_2$ is an $MV$-homomorphism.
It follows that this definition of $EMV$-morphisms in the class of $MV$-algebras coincides with the definition of $MV$-homomorphisms.

(3) Let $f: M_1\to M_2$ be an $MV$-homomorphism. For each $a\in\mI(M_1)$, let $f_a=f|_{_{[0,a]}}$, then $h=\{f_a\mid a\in \mI(M_1)\}:M_1\ra M_2$ is an $EMV$-morphism.

(4) If $S$ is a full subset of  $\mathcal I(M_1)$,  then $\{f_a \mid a \in S\}$ is an $EMV$-morphism.

(5) Let $f:M_1\to M_2$ be an $MV$-homomorphism. If $\{b_j\mid j\in J\}$ is a full subset of $\mI(M_2)$ and $g_j:M_1\ra
[0,b_j]$ is defined by $g_j(x):=f(x)\wedge b_j$, then $\{g_j\mid j\in J\}:M_1\ra M_2$ is an $EMV$-morphism (for more details, see
the proof of Theorem \ref{9.7}).
\end{exm}

\begin{rmk}\label{9.1.1}
Every strong $EMV$-homomorphism of $EMV$-algebras $f:M_1\ra M_2$ can be viewed as an $EMV$-morphism. Indeed,
$\mI(M_1)$ is a full subset of $M_1$ and $\{f(a)\mid a\in \mI(M_1)\}$ is a full subset of $M_2$ (by definition). For each $a\in\mI(M_1)$, set $f_a=f|_{_{[0,a]}}$. Then it can be easily seen that $\{f_a\mid a\in\mI(M_1)\}$ is an $EMV$-morphism, which we denote simply by $f$.	
Moreover, if $a_1,a_2$ are idempotents of $M_1$, then $f_{a_1}(x)=f_{a_2}(x)$ for each $x \in [0,a_1\wedge a_2]$.

Conversely, let $f=\{f_a\mid a \in \mathcal I(M_1)\}$ be an $EMV$-morphism such that for all $a_1,a_2\in \mathcal I(M_1)$, $f_{a_1}(x)=f_{a_2}(x)$ for each $x\in [0,a_1\wedge a_2]$. Then there is a unique strong $EMV$-homomorphism $f_0$ such that $f_a=f_0|_{_{[0,a]}}$ for each $a\in \mathcal I(M_1)$.

Indeed, let $x\in M$ be given. There are idempotents $a_1,a_2\in M_1$. Since $f_{a_1}(x)=f_{a_2}(x)$ for each $x \in [0,a_1\wedge a_2]$, we can define a mapping $f_0:M_1 \to M_2$ by $f_0(x)=f_a(x)$ whenever $x\le a\in \mathcal I(M_1)$. Due to the hypothesis, $f_0$ is defined unambiguously. Now, given $x,y\in M_1$, there is $a \in \mathcal I(M_1)$ such that $x,y\le a$. Hence, for all $z\in [0,a]$, we have $f_0(z)=f_a(z)$ which shows that $f_0$ preserves $\oplus,\vee$ and $\wedge$. In addition, $f_0(\lambda_a(x))=f_a(\lambda_a(x))=\lambda_{f_a(a)}(f_a(x))= \lambda_{f_0(a)}(f_0(x))$, so that $f_0$ is an $EMV$-homomorphism which is also strong. Moreover, $f_a=f_0|_{_{[0,a]}}$ for each $a\in \mathcal I(M_1)$.
\end{rmk}

\begin{exm}\label{exm3.4}
	Let $(M_i;\oplus,',0_i,1_i)$ and $(N_i;\oplus,',\mathsf{0}_i,\mathsf{1}_i)$ be $MV$-algebras for all $i\in\mathbb{N}$, and  $\{f_i:M_i\ra N_i\mid i\in\mathbb{N}\}$ be a family of homomorphisms between $MV$-algebras. For each finite subset $I\s \mathbb{N}$, define $a_I=(a_i)_{i\in\mathbb{N}}$ by (1) $a_i=1_i$, if $i\in I$ and (2) $a_i=0$, if $i\in\mathbb{N}\setminus I$. Clearly, $\{a_I\mid I\in S\}$ is a full subset of the $EMV$-algebra $\sum_{i\in\mathbb{N}} M_i$, see \cite[Ex 3.2(6)]{Dvz}, where $S$ is the set of all finite subsets of $\mathbb{N}$.
In a similar way, we define a full subset $\{b_I\mid I\in S\}$ of the $EMV$-algebra $\sum_{i\in\mathbb{N}} N_i$. For each $I\in S$, we define $f_I:[0,a_I]\ra [0,b_I]$, by $f_I((x_i)_{i\in\mathbb{N}})=(c_i)_{i\in I}$, where $c_i=f_i(x_i)$ for all $x\in I$, and $c_i=\mathsf{0}_i$ otherwise.	 It can be easily seen that $f:=\{f_I\mid I\in S\}: \sum_{i\in\mathbb{N}} M_i\ra \sum_{i\in\mathbb{N}} N_i$ is an	$EMV$-morphism.
\end{exm}

\begin{exm}
	Let $B$ be the set of all finite subsets of $\mathbb{N}$. Then $(B;\cup,\cap,\emptyset)$ is a generalized Boolean algebra.
	For each $i\in\mathbb{N}$, define $A_i=\{1,2,\ldots,i\}$. Clearly, $\{A_i\mid i\in \mathbb{N}\}$ is a full subset
	of $B$. Consider the family $f:=\{f_i:[\emptyset,A_i]\ra [\emptyset,A_{i-1}]\mid i\in\mathbb{N}\}$, where
	$f_i(X)=X\backslash \{i\}$ for all $X\s A_i$.
	Clearly, $\{f_i(A_i)\mid i\in\mathbb{N}\}$ is a full subset of $B$, too.
	It can be easily checked that $f$ is an $EMV$-morphism.
\end{exm}

As we have seen in Example \ref{ex:MV}, if $f:M_1\to M_2$ is an $MV$-homomorphism of two $MV$-algebras $M_1,M_2$, then each restricted map
$f_a:=f|_{_{[0,a]}}:[0,a]\ra [0,f(a)]$ is an $MV$-homomorphism, too. Hence, for every full subset $B$ of $\mI(M_1)$, the set $\{f_a\mid a\in B\}$ is an $EMV$-morphism from $M_1$ to $M_2$, so we can find many $EMV$-homomorphisms induced by $f$. Of course, all of them are considered as $EMV$-morphisms. Therefore, we propose an equivalence relation, called similarity and denoted by $\approx$, on the set $\Hom(M_1,M_2)$ of all $EMV$-morphisms from an $EMV$-algebra $M_1$ into an $EMV$-algebra $M_2$.

\begin{defn}\label{hhhh}
Let $M_1$ and $M_2$ be $EMV$-algebras and let $f=\{f_i\}_{i\in I}:M_1\ra M_2$ and $g=\{g_j\}_{j\in J}:M_1\ra M_2$  be two $EMV$-morphisms
such that $\textsf{e}(f)=\{a_i\mid i\in I\}$ and $\textsf{e}(g)=\{b_j\mid j\in J\}$.
   	We say that they are {\it similar} (denoted by $f\approx g$) if, for each $i\in I$, there exists $j\in J$
   	such that $a_i\leq b_j$ and
   	\[f_i(x)=g_j(x)\wedge f_i(a_i),\quad \forall x\in [0,a_i]. \]
   \end{defn}
In the next proposition we will establish some properties of similar $EMV$-morphisms.

\begin{prop}\label{eq.property}
   	Let $f=\{f_i\}_{i\in I}:M_1\ra M_2$ and $g=\{g_j\}_{j\in J}:M_1\ra M_2$  be two similar $EMV$-morphisms
   	such that $\textsf{e}(f)=\{a_i\mid i\in I\}$ and $\textsf{e}(g)=\{b_j\mid j\in J\}$. Then for all $i\in I$ and
   	$k\in J$, we have
   	\begin{itemize}[nolistsep]
   	\item[{\rm (i)}] if $f_i(a_i)\leq  g_k(b_k)$, then for all $x\in [0,a_i\wedge b_k]$, $f_i(x)=g_k(x)\wedge f_i(a_i)$;
   	
   	\item[{\rm (ii)}] if $f_i(a_i)=g_k(b_k)$, then for all $x\in [0,a_i\wedge b_k]$, $f_i(x)=g_k(x)$;
   	
   	\item[{\rm (iii)}] if $I'\s I$ is such that for all $i\in I$ there exists $k\in I'$ with $a_i\leq a_k$ and $f_i(a_i)\leq f_k(a_k)$, then
   	$h:=\{f_i\mid i\in I'\}$ is an $EMV$-morphism from $M_1$ to $M_2$ which is similar to $f$.
   \end{itemize}
   \end{prop}

   \begin{proof}
	(i) Let $i\in I$ and $k\in J$ such that $f_i(a_i)\leq g_k(b_k)$. Since $f\approx g$, then there exists $j\in J$ such that $a_i\leq b_j$,
	$f_i(a_i)\leq g_j(b_j)$ and for all $x\in [0,a_i]$ we have $f_i(x)=g_j(x)\wedge f_i(a_i)$. Put $t\in J$ such that
	$b_t\geq b_k,b_j$ and $g_t(b_t)\geq g_k(b_k),g_j(b_j)$. Then
	\begin{equation*}
	\big( g_k(x)=g_t(x)\wedge g_k(b_k),\ \forall x\in [0,b_k] \big) \quad \& \quad  \big( g_j(x)=g_t(x)\wedge g_j(b_j),\ \forall x\in [0,b_j]  \big).
	\end{equation*}
	Hence, for all $x\in [0,a_i\wedge b_k]$ we have
	\begin{eqnarray*}
		g_k(x)\wedge f_i(a_i)= g_t(x)\wedge g_k(b_k)\wedge f_i(a_i)=g_t(x)\wedge f_i(a_i)=g_t(x)\wedge g_j(b_j)\wedge f_i(a_i)
		= g_j(x)\wedge f_i(a_i)=f_i(x).
	\end{eqnarray*}

    (ii) Let $i\in I$ and $k\in J$ be such that $f_i(a_i)=g_k(b_k)$. Then by (i), for all $x\in [0,a_i\wedge b_k]$, we have
    \[f_i(x)=g_k(x)\wedge f_i(a_i)=g_k(x)\wedge g_k(b_k)=g_k(x).\]

    (iii) Clearly, $\{a_i\mid i\in I'\}$ and $\{f_i(a_i)\mid i\in I'\}$ are full subsets of $\mathcal I(M_1)$ and $\mathcal I(M_2)$, respectively. Moreover,
    the other conditions of Definition \ref{9.1} hold evidently. Thus $h:=\{f_i\mid i\in I'\}:M_1\ra M_2$ is an $EMV$-morphism.
    Now, we show that $f\approx h$. Let $i\in I$. Then by the assumption, there exists $j\in I'$ such that
    $a_i\leq a_j$ and $f_i(a_i)\leq f_j(a_j)$. Since $f$ is an $EMV$-morphism, then for each $x\in [0,a_i]$ we have
    $f_i(x)=f_j(x)\wedge f_i(a_i)$. Therefore, $f\approx h$.
   \end{proof}

   \begin{prop}\label{eq.prop}
   	The relation $\approx$ is an equivalence relation on the set $\Hom(M_1,M_2)$ of all $EMV$-morphisms from an $EMV$-algebra $M_1$ into an $EMV$-algebra $M_2$.    	
   \end{prop}

   \begin{proof}
   Let $M_1$ and $M_2$ be two $EMV$-algebras.
   	Clearly, $\approx$ is reflexive.
   	Let $f=\{f_i\mid i\in I\}:M_1\ra M_2$ and $g=\{g_j\mid j\in J\}:M_1\ra M_2$  be two $EMV$-morphisms such that $f\approx g$.
   	Put $j\in J$. Since $\mathsf{e}(f)$ is a full subset of $M_1$, then there exists $i\in I$ such that $b_j\leq a_i$. Similarly, since
   	$\{f_i(a_i)\mid i\in I\}$ is a full subset of $M_2$, then there is $z\in I$ such that $f_z(a_z)\geq g_j(b_j)$. So, we can find $w\in I$ with
   	$a_i,a_z\leq a_w$ and $f_i(a_i),f_z(a_z)\leq f_w(a_w)$. Without loss of generality we assume that $a_i=a_w$. Then
   	\begin{equation}
   	\label{eq.prop.R1} b_j\leq a_i,\quad g_j(b_j)\leq f_i(a_i).
   	\end{equation}
   	From $f\approx g$ it follows that there exists $s\in J$ such that $b_s\geq a_i$, $f_i(a_i)\leq g_s(b_s)$ and
   	\begin{equation}
   	\label{eq.prop.R2} \forall x\in [0,a_i],\quad f_i(x)=g_s(x)\wedge f_i(a_i).
   	\end{equation}
   	Since $g_j(b_j)\leq g_s(b_s)$, then
   	\begin{equation}
   	\label{eq.prop.R3} \forall x\in [0,b_j\wedge b_s]=[0,b_j],\quad g_j(x)=g_s(x)\wedge g_j(b_j).
   	\end{equation}
   	and so for each $x\in [0,b_j]$ we have
   	\begin{eqnarray*}
   	f_i(x)\wedge g_j(b_j)&=& (g_s(x)\wedge f_i(a_i))\wedge g_j(b_j), \quad \mbox{ by (\ref{eq.prop.R2}) } \\
   	&=& (g_s(x)\wedge g_j(b_j))\wedge f_i(a_i)=g_j(x)\wedge f_i(a_i), \quad \mbox{ by (\ref{eq.prop.R3}) } \\
   	&=& (g_j(x)\wedge g_j(b_j))\wedge f_i(a_i), \quad \mbox{  since $x\leq b_j$} \\
   	&=& g_j(x)\wedge (g_j(b_j)\wedge f_i(a_i))=g_j(x)\wedge g_j(b_j), \quad \mbox{ by (\ref{eq.prop.R1}) } \\
   	&=& g_j(x).
   	\end{eqnarray*}
   	Hence, $\approx$ is symmetric.
   	
   	To prove that $\approx$ is transitive, we let $h=\{h_i\mid i\in K\}:M_1\ra M_2$ be an $EMV$-morphism
   	with $\mathsf{e}(h)=\{c_i\mid i\in K\}$ such that $g\approx h$.   Put $i\in I$. Then
   	\begin{eqnarray}
   	f\approx g\ \Rightarrow \exists\, j\in J: a_i\leq b_j,\ f_i(a_i)\leq g_j(b_j)\ \ \& \ \ \forall x\in [0,a_i] \ f_i(x)=g_j(x)\wedge f_i(a_i). \\
   	g\approx h \ \Rightarrow \exists\, k\in K: b_j\leq c_k,\ g_j(b_j)\leq h_k(c_k)\ \ \& \ \  \forall x\in [0,b_j] \ g_j(x)= h_k(x)\wedge g_j(b_j).
   	\end{eqnarray}
   	and so for all $x\in [0,a_i]$ we have
   	\[f_i(x)=g_j(x)\wedge f_i(a_i)=h_k(x)\wedge g_j(b_j)\wedge f_i(a_i)=h_k(x)\wedge f_i(a_i).\]
   	That is, $f\approx h$. Therefore, $\approx$ is an equivalence relation.
   \end{proof}

Now we introduce a composition  of two $EMV$-morphisms.

\begin{prop}\label{9.2}
	Let $M_1$, $M_2$ and $M_3$ be $EMV$-algebras and $f=\{f_i\mid i\in I\}:M_1\ra M_2$ and
	$h=\{h_j\mid j\in J\}:M_2\ra M_3$ be 	 $EMV$-morphisms with $\mathsf{e}(f)=\{a_i\mid i\in I\}$
	and $\mathsf{e}(h)=\{b_j\mid j\in J\}$. If  $U=\{(i,j)\in I\times J\mid f_i(a_i)\leq b_j\}$. For each $(i,j)\in U$, we define
	$g_{i,j}:[0,a_i]\ra [0,h_j\circ f_i(a_i)]$ by $g_{i,j}=h_j\circ f_i$.
	Then $g=\{g_{i,j}\mid (i,j)\in U\}$ is an $EMV$-morphism, which is denoted $h\circ f$.
\end{prop}

\begin{proof}
Clearly, $\mathsf{e}(g)=\{a_i\mid (i,j)\in U\}$ is a full subset of $M_1$ and $\{g_{i,j}(a_i)\mid (i,j)\in U\}\s \mI(M_3)$. Put $z\in M_3$.
Then there is $j\in J$ such that $h_j(b_j)\geq z$. Similarly, there exist $i\in I$ and $t\in J$ such that $f_i(a_i)\geq b_j$,
$f_i(a_i)\leq b_t$, and $h_t(b_t)\geq h_j(b_j)$. Thus $h_t(f_i(a_i))\geq h_t(b_j)\geq h_j(b_j)\geq z$.
That is,  $\{g_{i,j}(a_i)\mid (i,j)\in U\}$ is a full subset of $\mI(M_3)$. Now, let $(i,j),(s,t)\in U$ such that
$(h_j\circ f_i)(a_i)\leq (h_t\circ f_s)(a_s)$. We claim that
\begin{equation}\label{9.2r1}
\forall x\in [0,a_i\wedge a_s]\quad (h_j\circ f_i)(x)=(h_t\circ f_s)(x)\wedge (h_j\circ f_i)(a_i).
\end{equation}
Let $z\in I$ such that $a_z\geq a_i,a_s$ and $f_z(a_z)\geq f_i(a_i),f_s(a_s)$. Then
\begin{eqnarray*}
\exists\, k_1\in J: b_{k_1}\geq f_z(a_z) &\Rightarrow& \exists\, k_2\in J : b_{k_2}\geq b_{k_1},b_t\  \&  \
h_{k_2}(b_{k_2})\geq h_{k_1}(b_{k_1}),h_t(b_t)  \\
&\Rightarrow & \exists\, k\in J: b_k\geq b_j,b_t \ \& \ h_k(b_k)\geq h_j(b_j),h_t(b_t) \\
&\Rightarrow & \Big(h_j(x)=h_k(x)\wedge h_j(b_j),  x\in [0,b_j]\Big) \ \& \
\Big(h_t(x)=h_k(x)\wedge h_t(b_t), x\in [0,b_t]\Big).
\end{eqnarray*}
It follows that for each $x\in [0,a_i\wedge a_j]$ we have
\begin{eqnarray*}
	h_t(f_s(x))\wedge h_j(f_i(a_i))&=& h_k(f_s(x))\wedge h_t(b_t)\wedge h_k(f_i(a_i))\wedge h_j(b_j)\\
&=&	h_k(f_s(x)\wedge f_i(a_i))\wedge h_t(b_t)\wedge h_j(b_j) \\
	&=& h_k(f_z(x)\wedge f_s(a_s)\wedge f_i(a_i))\wedge h_t(b_t)\wedge h_j(b_j)\\
&=& h_k(f_i(x)\wedge f_s(a_s))\wedge h_t(b_t)\wedge h_j(b_j)\\
	&=& h_k(f_i(x))\wedge h_k(f_s(a_s))\wedge h_t(b_t)\wedge h_j(b_j) \\
	&=& h_k(f_i(x))\wedge h_t(f_s(a_s))\wedge h_j(b_j),\mbox{  since $b_t\geq f_s(a_s)$} \\
	&=& h_j(f_i(x))\wedge h_t(f_s(a_s)) \\
	&=& h_j(f_i(x)),\mbox{  since $h_j(f_i(x))\leq h_j(f_i(a_i))\leq h_t(f_s(a_s))$}.
\end{eqnarray*}
Thus (\ref{9.2r1}) holds.
Finally, let $(i,j)$ and $(s,t)$ be arbitrary elements of $U$. Then by the assumption, we can find $z\in I$ such that
$a_i,a_s\leq a_z$ and $f_i(a_i),f_s(a_s)\leq f_z(a_z)$. Put $k\in J$ with
$b_k\geq f_z(a_z),b_j,b_t$ and $h_k(b_k)\geq h_j(b_j),h_t(b_t)$.
Since $h_k(b_k)\geq h_j(b_j),h_t(b_t)$, then
\[h_j(f_i(a_i))=h_k(f_i(a_i))\wedge h_j(b_j)\quad \& \quad h_t(f_s(a_s))=h_k(f_s(a_s))\wedge h_t(a_t)\]
and so
$h_k(f_z(a_z))\geq h_k(f_i(a_i))\geq h_j(f_i(a_i))$. In a similar way, $h_k(f_z(a_z))\geq h_k(f_i(a_i))\geq h_t(f_s(a_s))$. Hence,
there exists $(z,k)\in U$ such that $a_z\geq a_i,a_s$ and
$(h_k\circ f_z)(a_z)\geq (h_j\circ f_i)(a_i), (h_t\circ f_s)(a_s)$.
Therefore, $h\circ f$ is an $EMV$-morphism.
\end{proof}

The next proposition shows that the equivalence relation $\approx$ is compatible with the composition $\circ$.

\begin{prop}\label{9.2.1}
Let $f=\{f_i\}_{i\in I}:M_1\ra M_2$ and $g=\{g_j\}_{j\in J}:M_1\ra M_2$  be
two $EMV$-morphisms, where
$\e(f)=\{a_i\mid i\in I\}$ and $\e(g)=\{b_j\mid j\in J\}$ such that
$f \approx g$.
\begin{itemize}
\item[{\rm (i)}] If $h=\{h_k\mid k\in K\}:M_2\ra M_3$ is an $EMV$-morphism
with $\e(h)=\{c_i\mid i\in K\}$, then $h\circ f\approx h\circ g$.
\item[{\rm (ii)}] If $v=\{v_k\mid k\in K\}:M_0\ra M_1$ is an
$EMV$-morphism with $\e(v)=\{d_i\mid i\in K\}$, then $f\circ
v\approx g\circ v$.
\item[{\rm(iii)}] If $g_1,g_2\in \Hom(M_1,M_2)$, $g_1\approx g_2$, and $f_1,f_2 \in \Hom(M_2,M_3)$, $f_1\approx f_2$, then $f_1\circ g_1 \approx f_2\circ g_2$.
\end{itemize}
\end{prop}

\begin{proof}
(i) Set $U=\{(i,k)\in I\times K\mid f_i(a_i)\leq c_k\}$ and $V=\{(j,k)\in
J\times K\mid g_j(b_j)\leq c_k\}$. Then
\[h\circ f=\{h_k\circ f_i\mid (i,k)\in U\}\quad \& \quad h\circ
g=\{h_k\circ g_j\mid (j,k)\in V \}.\]
We  show that for each $(i,k)\in U$, there exists $(s,t)\in V$ such that
$a_i\leq b_s$,
$(h_k\circ f_i)(a_i)\leq (h_t\circ g_s)(b_s)$ and
\[\forall x\in [0,a_i]\ \  (h_k\circ f_i)(x)=(h_t\circ g_s)(x)\wedge
(h_k\circ f_i)(a_i). \]
Put $(i,k)\in U$. Then $f_i(a_i)\leq c_k$. Since $f\approx g$, given $i$ there is $s\in J$ such that $a_i\leq b_s$, $f_i(a_i)\leq g_s(b_s)$ and
\[\forall x\in [0,a_i]\quad   f_i(x)=g_s(x)\wedge f_i(a_i).\]
Let $t\in K$ such that $c_k,g_s(b_s)\leq c_t$ and $h_k(c_k)\leq h_t(c_t)$.
Then
\begin{equation}
\label{9.2.1R1} h_k(f_i(a_i))\leq h_t(f_i(a_i))\leq h_t(g_s(b_s)),
\end{equation}
whence $h_k(z)=h_t(z)\wedge h_k(c_k)$ for all $z\in [0,c_k]$
and so for all $x\in [0,a_i]$, we have
$h_k(f_i(x))=h_t(f_i(x))\wedge h_k(c_k)=h_t(g_s(x)\wedge f_i(a_i))\wedge
h_k(c_k)=h_t(g_s(x))\wedge h_t(f_i(a_i))\wedge h_k(c_k)=
h_t(g_s(x))\wedge h_k(f_i(a_i))$. That is, $h\circ f\approx h\circ g$.

(ii) Set $W:=\{(k,i)\in K\times I\mid v_k(d_k)\leq a_i\}$ and
$Z:=\{(k,j)\in K\times J\mid v_k(d_k)\leq b_j\}$.
Then $f\circ v=\{f_i\circ v_k\mid (k,i)\in W\}$, $g\circ
v=\{g_j\circ v_k\mid (k,j)\in Z\}$,
$\e(f\circ v)=\{d_k\mid (k,i)\in W\}$ and $\e(g\circ v)=\{d_k\mid
(k,j)\in Z\}$. It suffices to show that
for each $(k,i)\in W$, there exists $(s,j)\in Z$ such that $d_k\leq d_s$,
$(f_i\circ v_k)(d_k)\leq (g_j\circ v_s)(d_s)$ and
for all $u\in [0,d_k]$, $(f_i\circ v_k)(u)=(g_j\circ
v_s)(u)\wedge (f_i\circ v_k)(d_k)$.
Let $(k,i)\in W$. Since $f\approx g$, there is $j\in J$ such that $a_i\leq
b_j$ and for all $x\in [0,a_i]$, we have
$f_i(x)=g_j(x)\wedge f_i(a_i)$. Set $s=k$. Then for each $u\in [0,d_k]$,
$v_k(u)\in [0,a_i]$ and
$f_i(v_k(u))=g_j(v_k(u))\wedge f_i(a_i)$. Since $u\leq d_k\leq
a_i$ and $v_k(d_k) \le a_i$, then $f_i(v_k(u))\leq f_i(v_k(d_k))\leq f_i(a_i)$ and so
\[f_i(v_k(u))=g_j(v_k(u))\wedge f_i(a_i)\wedge
f_i(v_k(d_k))=g_j(v_k(u))\wedge f_i(v_k(d_k)).\]
That is, $(f_i\circ v_k)(u)=(g_j\circ v_s)(u)\wedge (f_i\circ
v_k)(d_k)$ for each $u\in [0,d_k]$. Therefore, $f\circ v\approx g\circ v$.

(iii) By (i) and (ii), we have $f_1\circ g_1 \approx f_1 \circ g_2 \approx f_2\circ g_2$.
\end{proof}

In the following proposition we show that the operation $\circ$ on the set of all $EMV$-morphisms is $\approx$-associative, i.e. $h\circ(g\circ f)\approx (h\circ g)\circ f$ whenever the corresponding compositions are defined.

\begin{prop}\label{pr:assoc}
The operation $\circ$ is $\approx$-associative. If $Id_{M_1}:M_1\ra M_1$, $Id_{M_2}:M_2\ra M_2$, and $f \in \Hom(M_1,M_2)$, then $f\circ Id_{M_1}\approx f$ and $Id_{M_2}\circ f\approx f$.
\end{prop}

\begin{proof}
(1) Let $f=\{f_i\mid i \in I\}:M_1\ra M_2$, $g=\{g_j\mid j \in J\}:M_2\ra M_3$ and $h=\{h_k\mid k \in K\}:M_3\ra M_4$ be $EMV$-morphisms with $\mathsf{e}(f)=\{a_i\mid i\in I\}$,
$\mathsf{e}(g)=\{b_j\mid j\in J\}$ and $\mathsf{e}(h)=\{c_k\mid k\in K\}$. Then
$g\circ f=\{g_j\circ f_i\mid (i,j)\in I\times J,\ f_i(a_i)\leq b_j\}$,
$h\circ g=\{h_k\circ g_j\mid (j,k)\in J\times K,\ g_j(b_j)\leq c_k\}$,
$\mathsf{e}(g\circ f)=\{a_i\mid (i,j)\in I\times J,\ f_i(a_i)\leq b_j\}$
and $\mathsf{e}(h\circ g)=\{b_j\mid (j,k)\in J\times K,\ g_j(b_j)\leq c_k\}$. Hence,
$h\circ(g\circ k)=\{h_k\circ (g_j\circ f_i)\mid i\in I,\ j\in J,\ k\in K,\ (g_j(f_i(a_i)))\leq c_k,\ f_i(a_i)\leq b_j\}$ and
$(h\circ g)\circ f=\{(h_k\circ g_j)\circ f_i\mid i\in I,\ j\in J,\ k\in K,\ f_i(a_i)\leq b_j,\ g_j(b_j)\leq c_k\}$.
Assume that $f_i(a_i)\leq b_j$  and $g_j(b_j)\leq c_k$. Then clearly $g_j(f_i(a_i))\leq c_k$ and, for each $x\in [0,a_i]$, we have
$((h_k\circ g_j)\circ f_i)(x)=(h_k\circ(g_j\circ f_i))(x)=(h_k\circ(g_j\circ f_i))(x)\wedge ((h_k\circ g_j)\circ f_i)(a_i)$.
So, by Definition \ref{hhhh}, $h\circ(g\circ f)\approx (h\circ g)\circ f$.

(2) Now, consider the identity maps $Id_{M_1}:M_1\ra M_1$ and $Id_{M_2}:M_2\ra M_2$. By Remark \ref{9.1.1}, they are $EMV$-morphisms.
So that, $Id_{M_1}=\{Id_{[0,a]}\mid a\in \mI(M_1)\}$ and $U=\{(a,i)\in \mathcal I(M_1)\times I\mid Id_{[0,a]}(a)\le b_i\}=\{(a,i)\mid a\le b_i\}$. For each $(a,i)\in U$, let us put $f_{a,i}=f_i\circ Id_{[0,a]}=f_i: [0,a]\to [0,f_i(a)]$. Then for the $EMV$-morphism $f\circ Id_{M_1}=\{f_{a,i}\mid (a,i)\in U\}$, we have that for each $(a,i)$, that is for $a$, there is $b_i$ such that $a\le b_i$. Whence, $f_i(x)= f_{a,i}(x):=f_i(x)\wedge f_{a,i}(a)=f_i(x)$ for each $x \in [0,a]$, that is $f\circ Id_{M_1}  \approx f$.

Similarly, let $Id_{M_2} =\{Id_{[0,b]}\mid b \in \mathcal I(M_2)\}$. Then for $V=\{(i,b)\in I\times \mathcal I(M_2)\mid f_i(b_i)\le b\}$, we have $Id_{M_2}\circ f=\{\tilde f_{i,b}\mid (i,b)\in V\}$, where
$ f_i=\tilde f_{i,b}:= Id_{[0,b]}\circ f_i:[0,b_i]\to [0,f_i(b_i)]$. So that, $f_i(y)=\tilde f_{i,b}(y)=f_i(y)\wedge \tilde f_{i,b}(b_i)=f_i(y)$ for each $y \in [0,b_i]$, that is $Id_{M_2}\circ f \approx f$.
\end{proof}

\section{Categories of $EMV$-algebras}

The notion of an $EMV$-morphisms and the relation $\approx$ enable us to study three categories of $EMV$-algebras where the objects are $EMV$-algebras and morphisms are special classes of $EMV$-morphisms.

\begin{rmk}\label{rmk:category}
Let $\mathcal{EMV}$ be the set of $EMV$-algebras. Given two $EMV$-algebras $M_1,M_2$, on the set $\Hom(M_1,M_2)$ of $EMV$-morphisms from $M_1$ into $M_2$, we define $\Hom(M_1,M_2)/\approx:=\{[f]\mid f \in \Hom(M_1,M_2)\}$ of quotient classes, where $[f]=\{g \in \Hom(M_1,M_2)\mid g\approx f\}$. Since $\approx$ is an equivalence, we can define a composition of quotient classes denoted also by $\circ$ such that $[g]\circ [h]:= [g\circ h]$ for all $[h]\in \Hom(M_1,M_2)/\approx$ and $[g]\in \Hom(M_2,M_3)/\approx$. Then $[g]\circ [h]\in \Hom(M_1,M_3)/\approx$, and
by Proposition \ref{pr:assoc}, we have $([f]\circ [g])\circ [h]=[f]\circ ([g]\circ [h])$ for each $[f]\in \Hom(M_3,M_4)/\approx$.  In addition, $[h]\circ [Id_{M_1}]=[h]$ and $[Id_{M_2}]\circ [h]=[h]$ for each $[h]\in \Hom(M_1,M_2)/\approx$.

Therefore, we can define the category $\mathbb{EMV}$ whose objects are $EMV$-algebras and any morphism  from an object $M_1$ to another object $M_2$ is the quotient class $[f]\in \Hom(M_1,M_2)/\approx$. As we have seen, the category $\mathbb{EMV}$ is defined correctly.
\end{rmk}

\begin{defn}\label{standard}
	An $EMV$-morphism $f=\{f_i\mid i\in I\}:M_1\ra M_2$ with $\mathsf{e}(f)=\{a_i\mid i\in I\}$
	is called {\em standard} if, for each $x\in M_1$, the set $\{f_i(x)\mid x\leq a_i\}$ has a maximum element, that is, given $x \in M_1$, there is $a_j\in \mathcal I(M_1)$ with $j\in I$ and $a_j\ge x$ such that $f_{j}(x)=\max\{f_i(x)\mid x\leq a_i\}$.
\end{defn}

Clearly, if $M_1$ and $M_2$ are $MV$-algebras and $f:M_1\ra M_2$ is an $MV$-homomorphism, then $f$ is standard as an $EMV$-morphism.

\begin{prop}\label{cat-stan}
	Let $f=\{f_i\mid i\in I\}:M_1\ra M_2$ and $g=\{g_j\mid j\in J\}:M_2\ra M_3$ be standard $EMV$-morphisms such that
	$\mathsf{e}(f)=\{a_i\mid i\in I\}$ and $\mathsf{e}(g)=\{b_j\mid j\in J\}$. Then:
\begin{itemize}[nolistsep]	
\item[{\rm (i)}]  The $EMV$-morphism $g\circ f$ is standard.

\item[{\rm (ii)}] For each $EMV$-algebra $M$, the identity map $Id_M:M\ra M$ is a standard $EMV$-morphism.
	
\item[{\rm (iii)}] If $h=\{h_i\mid i\in K\}:M_1\ra M_2$ is an $EMV$-morphism such that $\mathsf{e}(h)=\{c_k\mid k\in K\}$ and
	$f\approx h$, then $h$ is a standard $EMV$-morphism.
\end{itemize}
\end{prop}

\begin{proof}
(i) We know that
	$g\circ f=\{g_j\circ f_i\mid (i,j)\in U\}$, where $U=\{(i,j)\in I\times J\mid f_i(a_i)\leq b_j\}$.
	Let $x\in M_1$ and $i\in I$ be such that $f_i(a_i)\leq b_j$ and
	$x\leq a_i$. Then there exists $s\in I$ with $f_s(x)=\max\{f_i(x)\mid x\leq a_i\}$. Suppose that
	$g_t(f_s(x))=\max\{g_j(f_s(x))\mid f_s(x)\leq b_j\}$. Put $k\in J$ such that $b_k\geq b_t,f_s(a_s)$ and
	$g_t(b_t)\leq g_k(b_k)$. Then $g_k(f_s(x))\geq g_t(f_s(x))$ which implies that $g_k(f_s(x))=g_t(f_s(x))$. We claim that
	$g_k(f_s(x))=\max\{(g_j\circ f_i)(x)\mid (i,j)\in U,\ x\leq a_i\}$. Let $(i,j)\in U$ and $x\leq a_i$. Then $f_i(x)\leq f_s(x)$.
	Let $z\in J$ such that $b_z\geq b_j,b_k$ and $g_z(b_z)\geq g_j(b_j),g_k(b_k)$. By the assumption,
	$g_z(f_s(x))=g_k(f_s(x))$ and $g_j(f_s(x))=g_z(f_s(x))\wedge g_j(b_j)\leq g_z(f_s(x))=g_k(f_s(x))$. So, the claim is true, and $g\circ f$ is a standard $EMV$-morphism.

(ii) It is evident.

(iii) If $x\in M_1$, then there
	exists $t\in I$ such that $x\leq a_t$ and $f_t(x)=\max\{f_i(x)\mid x\leq a_i\}$. Since $f\approx h$, there exists $k\in K$ such that
	$a_t\leq c_k$, $f_t(a_t)\leq h_k(c_k)$ and $f_t(x)=h_k(x)\wedge f_t(a_t)$ for all $x\in [0,a_t]$. That is, $f_t(x)\leq h_k(x)$. Moreover,
	if $c_j$ is an arbitrary element of $\mathsf{e}(h)$ such that $x\leq c_j$, then there is $s\in K$ such that
	$c_j,c_k\leq c_s$ and $h_j(c_j),h_k(c_k)\leq h_s(c_s)$, so $h_j(x),h_k(x)\leq h_s(x)$. On the other hand,
	there exists $i\in I$ such that $c_s\leq a_i$ and $h_s(x)=f_i(x)\wedge h_s(c_s)$, which implies that
	$h_s(x)\leq f_i(x)\leq f_t(x)$ (since $x\leq a_i$). Thus, $h_j(x)\leq f_t(x)$ (since $h_j(x)\leq h_s(x)$).
	Summing up the above results, we get that $f_t(x)=h_k(x)$ and $f_t(x)=h_s(x)=\max\{h_i(x)\mid i\in K\}$ and so, $h$ is a standard $EMV$-morphism.   	
\end{proof}

\begin{lem}\label{lem-stan1}
	Let $f=\{f_i\mid i\in I\}:M_1\ra M_2$ be a standard $EMV$-morphism with $\mathsf{e}(f)=\{a_i\mid i\in I\}$.
	\begin{itemize}[nolistsep]
		\item[{\rm (i)}] For each $x\in M_1$,
		$\max\{f_i(x)\mid x\leq a_i\}=\max\{f_i(x)\mid a\leq a_i\}$, where $a\in\mI(M_1)$ and $a\geq x$.
		\item[{\rm (ii)}] The map $F_f:M_1\ra M_2$ defined by $F_f(x)=\max\{f_i(x)\mid x\leq a_i\}$, $x \in M_1$, is a strong $EMV$-homomorphism in the sense of
		Definition {\rm \ref{3.4}}.
	\end{itemize}
\end{lem}

\begin{proof}
	(i) Clearly, $\max\{f_i(x)\mid a\leq a_i\}\leq \max\{f_i(x)\mid x\leq a_i\}$. Let $a\leq a_j$ and $f_k(x)$ be the greatest element of
	$\{f_i(x)\mid x\leq a_i\}$. Put $t\in I$ such that $a_j,a_k\leq a_t$ and $f_j(a_j),f_k(a_k)\leq f_t(a_t)$.
	Then $f_t(x)=f_k(x)$ and $f_t(x)\in \{f_i(x)\mid a\leq a_i\}$.
	It follows that $\max\{f_i(x)\mid x\leq a_i\}=\max\{f_i(x)\mid a_j\leq a_i\}$.
	
	(ii) Clearly, $F_f(0)=0$.
Let $x,y\in M_1$. If $F_f(x\oplus y)=f_k(x\oplus y)$ for $a_k\geq x\oplus y$,	 then $F_f(x\oplus y)=f_k(x)\oplus f_k(y)\leq F_f(x)\oplus F_f(y)$. On the other hand, if $F_f(x)=f_i(x)$ and $F_f(y)=f_j(y)$, then there is $s\in I$ such that $a_s\geq a_i,a_j$ and $f_s(a_s)\geq f_i(a_i),f_j(a_j)$, so that $f_s(x)=f_i(x)$ and $f_s(y)=f_j(y)$. Thus $F_f(x)\oplus F_f(y)=f_s(x)\oplus f_s(y)=f_s(x\oplus y)\leq F_f(x\oplus y)$
	(since $x\oplus y\leq a_s\oplus a_s=a_s$). That is, $F_f(x\oplus y)=F_f(x)\oplus F_f(y)$. In a similar way, we can show that
	$F_f(x\vee y)=F_f(x)\vee F_f(y)$ and $F_f(x\wedge y)=F_f(x)\wedge F_f(y)$. It suffices to show that for each $a\in\mI(M_1)$,
	$F_f|_{_{[0,a]}}:[0,a]\ra [0,F_f(a)]$ is an $MV$-homomorphism. Put $a\in\mI(M_1)$ and $x\in [0,a]$. Let $F_f(a)=f_j(a)$ and
	$F_f(x)=f_i(x)$ for $x\leq a_i$ and $a\leq a_j$.

	(1) If $k\in I$ such that $a_j,a_i\leq a_k$ and $f_i(a_i),f_j(a_j)\leq f_k(a_k)$, then
	$F_f(x)=f_k(x)$ and $F_f(a)=f_k(a)$ and so by Proposition \ref{3.5}, we have
	\begin{eqnarray*}
	 \lam_{F_f(a)}(F_f(x))&=&\lam_{f_k(a)}(f_k(x))=\lam_{f_k(a_k)}(f_k(x))\wedge f_k(a)=f_k(\lam_{a_k}(x))\wedge f_k(a)\\
	&=& f_k(\lam_{a_k}(x)\wedge a)=f_k(\lam_{a}(x)).
	\end{eqnarray*}
	
(2) Let $F_f(\lam_a(x))=f_t(\lam_a(x))$ for some $t\in I$ with $\lam_a(x)\leq a_t$.
	Put $z\in I$ such that $a_t,a_k\leq a_z$ and $f_k(a_k),f_t(a_t)\leq f_z(a_z)$. Then
	$F_f(\lam_a(x))=f_z(\lam_a(x))$ and $\lam_{F_f(a)}(F_f(x))=f_z(\lam_{a}(x))$ (by (1)). So,
	$\lam_{F_f(a)}(F_f(x))=F_f(\lam_a(x))$. Hence, $F_f$ is an $EMV$-homomorphism in the sense of Definition \ref{3.4}. Since $\{f_i(a_i)\mid i\in I\}$ is a full subset of $M_2$, then
	clearly, $\{F_f(a_i)\mid i\in I\}$ is a full subset of $M_2$, which implies that $F_f$ is a strong $EMV$-morphism.
\end{proof}

\begin{rmk}\label{cat0}
Given two $EMV$-algebras $M_1$ and $M_2$, let $\Hom_s(M_1,M_2)$ be the set of all standard $EMV$-morphisms from $M_1$ into $M_2$. We denote by $\Hom_s(M_1,M_2)/\approx$ the set of quotient classes $[f]_s$, where $f\in \Hom_s(M_1,M_2)$, under the equivalence $\approx$ restricted to $\Hom_s(M_1,M_2)$, that is,
$$\Hom_s(M_1,M_2)/\approx:=\{[f]_s\mid f \in \Hom_s(M_1,M_2)\},
$$
where $[f]_s:=\{g \in \Hom_s(M_1,M_2)\mid g\approx f\}$.
Due to Proposition \ref{cat-stan}, we can define unambiguously a composition, denoted also as $\circ$, of two classes via $[f]_s\circ [g]_s:=[f\circ g]_s$ for $[f]_s\in \Hom_s(M_1,M_2)/\approx$ and $[g]_s \in \Hom_s(M,M_1)/\approx$. Then $\circ$ is associative and due to Proposition \ref{cat-stan}(ii), we have $[h]_s\circ [Id_{M_1}]_s=[h]_s$ and $[Id_{M_2}]_s\circ [h]_s=[h]_s$ for each $[h]_s\in \Hom_s(M_1,M_2)/\approx$.

Therefore, similarly as in Corollary \ref{rmk:category}, we can define a new category $\mathbb{EMV}_s$ whose objects are $EMV$-algebras and a morphism from an object $M_1$ into another object $M_2$ is $[f]_s\in \Hom_s(M_1,M_2)/\approx$. Now, it is easy to verify that $\mathbb{EMV}_s$ is indeed a category.
\end{rmk}

\begin{rmk}\label{rm:strong}
Let $f:M_1\to M_2$ and $g:M_2\to M_3$ be strong $EMV$-homomorphisms. Then the composition $g\circ f:M_1\to M_3$ defined by $(g\circ f)(x)=g(f(x))$, $x \in M_1$, is a strong $EMV$-homomorphism from $M_1$ into $M_3$. Therefore, we can define a third category $\mathbb{EMV}_{sh}$ whose objects are $EMV$-algebras and a morphism from an object $M_1$ into another object $M_2$ is any strong $EMV$-homomorphism from $M_1$ into $M_2$. Since every $Id_M$ is also a strong $EMV$-homomorphism, we see that $\mathbb{EMV}_{sh}$ is indeed a category.
\end{rmk}

\begin{thm}\label{cat}
The categories $\mathbb{EMV}_{sh}$ and $\mathbb{EMV}_s$ are equivalent.
\end{thm}

\begin{proof}
First we note that if $f=\{f_i\mid i\in I\}:M_1\ra M_2$ and $g=\{g_j\mid j\in J\}:M_1\ra M_2$ are two similar standard $EMV$-morphisms, then
	$F_f = F_g$, where the strong $EMV$-homomorphisms $F_f$ and $F_g$ are defined by Lemma \ref{lem-stan1}. Indeed, let $x\in M_1$. Then there exists $i\in I$ and $j\in J$ such that
	$F_f(x)=f_i(x)$ and $F_g(x)=g_j(x)$. Since $f\approx g$, then there is $k\in J$ such that $a_i\leq b_k$ and
	$f_i(x)=g_k(x)\wedge f_i(a_i)$. Also, by Proposition \ref{eq.prop}, there exists $t\in I$ such that $b_j\leq a_t$ and
	$g_j(x)=f_t(x)\wedge g_j(b_j)$, so
	$f_i(x)\leq g_k(x)\leq g_j(x)\leq f_t(x)\leq f_i(x)$ (since $f_i(x)=\max\{f_s(x)\mid x\leq a_s\}$). That is, $F_f(x)= F_g(x)$ for each $x \in M$.

Consider a mapping $\mathcal{F}:\mathbb{EMV}_s\ra \mathbb{EMV}_{sh}$ defined by	$\mathcal{F}(M)=M$ for all $EMV$-algebras $M$, and $\mathcal{F}([f]_s):M_1\ra M_2$, $[f]_s\in \Hom(M_1,M_2)_s/\approx$, sending $x$ to $F_f(x)=\max\{f_i(x)\mid x\leq a_i\}$, that is, $\mathcal F([f]_s)=F_f$. By the previous paragraph, $\mathcal F([f]_s)$ is correctly defined and in view of	Proposition \ref{cat-stan}, $\mathcal{F}$ preserves both $\circ$ and the identity maps. Therefore, $\mathcal F$ is a well-defined functor.

Let $\mathcal{H}:\mathbb{EMV}_{sh}\ra \mathbb{EMV}_s$
be defined by $\mathcal{H}(M)=M$ for each $EMV$-algebra $M$ and $\mathcal{H}(f):M_1\ra M_2$ be the class of standard $EMV$-morphisms containing the $EMV$-morphism induced by the strong $EMV$-homomorphism $f$ (see Remark \ref{9.1.1}), that is $\mathcal H(f)=[f]_s$. Then $\mathcal H$ is also a functor.

Let $f=\{f_i\mid i\in I\}:M_1\ra M_2$ be a standard $EMV$-morphism with $\mathsf{e}(f)=\{a_i\mid i\in I\}$.
	We claim that $\mathcal{H}\circ\mathcal{F}([f]_s)=[f]_s$. Set $F_f=\mathcal{F}([f]_s)$. It suffices to show that $\{F_{f,a}\mid a\in \mI(M_1)\}$ is similar to
	$f$, where $F_{f,a}=F_f|_{_{[0,a]}}$ for all $a\in\mI(M_1)$. Put $a\in\mI(M_1)$. If $x$ is an arbitrary element of $[0,a]$, then
	there exists $i\in I$ such that $x\leq a_i$ and $F_{f,a}(x)=f_i(x)=\max\{f_s(x)\mid x\le a_s\}$. It follows that
	$F_{f,a}(x)=F_f(x)=f_i(x)=f_i(x)\wedge F_{f,a}(a)$ (note that $F_{f,a}(a)\geq F_f(x)=F_{f,a}(x)$). Thus $f\approx \{F_{f,a}\mid a\in \mI(M_1)\}$.

	On the other hand, if $g:M_1\ra M_2$ is a strong $EMV$-homomorphism, then $\mathcal{H}(g)=[g]_s=[\{g_a\mid a\in\mI(M_1)\}]_s$ (see Remark \ref{9.1.1}) and so
	$\mathcal{F}\circ\mathcal{H}(g)(x)=\max\{g_a(x)\mid a\in\mI(M_1),\ x\leq a\}=\max\{g(x)\mid x\leq a\}=g(x)$ for all $x\in M_1$.
	Therefore, $\mathbb{EMV}_s$ and $\mathbb{EMV}_{sh}$ are equivalent categories.
\end{proof}

The following proposition will be used to show that $\mathbb{EMV}_s$ is a subcategory of the category $\mathbb{EMV}$.

\begin{prop}\label{pr:[f]}
Let $f=\{f_i\mid i\in I\}\in \Hom(M_1,M_2)$ be a standard $EMV$-morphism and $g=\{g_j\mid j\in
J\}\in [f]$, where $\textsf{e}(f)=\{a_i\mid i\in I\}$ and
$\textsf{e}(g)=\{b_j\mid j\in J\}$. Then $g$ is a standard $EMV$-morphism.
\end{prop}

\begin{proof}
Given $x\in M_1$, let
$f_k(x)=\max\{f_i(x)\mid x\leq a_i\}$, where $x\leq a_k$. Let $b_j$ be an
arbitrary element of $\textsf{e}(g)$ such that
$x\leq b_j$. Since $f\approx g$, then there exists $i\in I$ such that
$b_j\leq a_i$ and
$g_j(x)=f_i(x)\wedge g_j(b_j)$. From $f_i(x)\leq f_k(x)$ it follows that
\[ \mbox{ if $j\in J$ such that $x\leq b_j$, then } g_j(x)\leq f_k(x).\]
On the other hand, there is $t\in J$ such that $a_k\leq b_t$ and
$f_k(x)=g_t(x)\wedge f_k(a_k)$, which implies that $f_k(x)\leq g_t(x)\leq
f_k(x)$. That is,
$g_t(x)=\max\{g_j(x)\mid x\leq b_j\}$. Therefore, $g$ is standard.
\end{proof}

\begin{cor}\label{co:subcateg}
The category $\mathbb{EMV}_s$ is a subcategory of the category $\mathbb{EMV}$.
\end{cor}

\begin{proof}
Let $f:M_1\to M_2$ be a standard $EMV$-morphism. Due to Proposition \ref{pr:[f]},
we have $[f]=[f]_s$ which gives the result.
\end{proof}

\section{Product of $EMV$-algebras in the Category $\mathbb{EMV}$}

We show that the product of a family of $EMV$-algebras with respect to $EMV$-morphisms can be defined also in the category $\mathbb{EMV}$ of $EMV$-algebras.

\begin{prop}\label{9.3}
	Let $M_1$ and $M_2$ be $EMV$-algebras and $f=\{f_i\mid i\in I\}:M_1\ra M_2$ be
	an $EMV$-morphism with $\mathsf{e}(f)=\{a_i\mid i\in I\}$. Consider a full subset $K\s \mI(M_1)$.
	For each $b\in K$, let us define $g_{i,b}:[0,b]\ra [0,f_i(b)]$ by $g_{i,b}:=f_{i}|_{_{[0,b]}}$, where $i\in I$ and $b\leq a_i$.
	Then $g=\{g_{i,b}\mid i\in I,\ b\leq a_i\}$ is an $EMV$-morphism and $f\approx g$.
\end{prop} 	

\begin{proof}
First we show that $g:M_1\ra M_2$ is an $EMV$-morphism.

	(1) Let $i\in I$. Then there is $b\in K$ such that $a_i\leq b$.
	Since $\mathsf{e}(f)$ is a full subset of $M_1$, we can find $j\in I$
	such that $b\leq a_j$. Put $t\in I$ satisfying
	$a_i,a_j\leq a_t$ and $f_i(a_i),f_j(a_j)\leq f_t(a_t)$. Then we have
	\[\big(\forall x\in [0,a_i] \  f_i(x)=f_t(x)\wedge f_i(a_i)\big) \quad \& \quad \big(\forall x\in [0,a_j] \ f_j(x)=f_t(x)\wedge f_j(a_j)\big).\]
	It follows that $f_t(b)\geq f_t(a_i)\geq f_i(a_i)$ and $f_t(b)\in \mathsf{e}(g)$.
	So, $\{g_{i,b}(b)\mid i\in I,\ b\in K,\ b\leq a_i\}$ is a full subset of
	$M_2$. Clearly, $K\s \mathsf{e}(g)$. Thus $\mathsf{e}(g)$ is a full subset of $M_1$.

	(2) Let $i,j\in I$ and $b,b'\in K$ such that $b\leq a_i$ and $b'\leq a_j$ and $g_{i,b}(b)\leq g_{j,b'}(b')$. Then
	$f_i(b)\leq f_j(b')$. Put $t\in I$ such that $a_i,a_j\leq a_t$, $f_i(a_i),f_j(a_j)\leq f_t(a_t)$. Hence,
	\[\big(\forall x\in [0,a_i] \ f_i(x)=f_t(x)\wedge g_i(a_i)\big) \quad \& \quad \big(\forall x\in [0,a_j] \ f_j(x)=f_t(x)\wedge f_j(a_j)\big)\]
	and so, for all $x\in [0,b]$,
	\[f_j(x)=f_t(x)\wedge f_j(a_j)=f_t(x)\wedge f_j(a_j)\wedge f_j(b)=f_t(x)\wedge f_j(b).\]
	We claim that for each $x\in [0,b\wedge b']$, $g_{j,b'}(x)\wedge g_{i,b}(b)=g_{i,b}(x)$.
	Let $x\in [0,b\wedge b']$. Then
	\begin{eqnarray*}
	f_j(x)\wedge f_i(a_i)&=&f_t(x)\wedge f_j(a_j)\wedge f_i(a_i)\\
	&=&f_t(x)\wedge f_i(a_i)\wedge f_i(b)\wedge f_j(a_j),\mbox{ since $f_t(x)\wedge f_i(a_i)=f_i(x)\leq f_i(b)$ }\\
	&=& f_t(x)\wedge f_i(a_i)\wedge f_i(b)\wedge f_j(a_j)\wedge f_j(b'),
	\mbox{ since $f_t(x)\wedge f_j(a_j)=f_j(x)$ and $f_j(x)\leq f_j(b')$} \\
	&=& f_t(x)\wedge f_i(b)\wedge f_j(b'), \mbox{ since $f_i(b)\leq f_i(a_i)$ and $f_j(b')\leq f_j(a_j)$ } \\
	&=& f_t(x)\wedge f_i(b),\mbox{  by the assumption} \\
	&=& f_t(x)\wedge f_i(b)\wedge f_i(a_i)=f_i(x)\wedge f_i(b)=f_i(x),
	\end{eqnarray*}
    whence $g_{j,b'}(x)\wedge g_{i,b}(b)=f_j(x)\wedge f_i(b)=f_j(x)\wedge (f_i(a_i)\wedge f_i(b))=f_i(x)\wedge f_i(b)=f_i(x)=g_{i,b}(x)$.

	(3) Let $b,b'\in K$ and $i,j\in I$ such that $b\leq a_i$ and $b'\leq a_j$. Then there is $b''\in K$ with $b,b'\leq b''$.
	Put $t\in I$ such that $b''\leq a_t$.
	\begin{eqnarray}
	\exists\, s\in I:\ a_i,a_j\leq a_s ,\ f_i(a_i)\leq f_s(a_s) ,\ f_j(a_j)\leq f_s(a_s)
	\end{eqnarray}
	similarly, there is $w\in I$ such that $a_t,a_s\leq a_w$ and $f_t(a_t),f_s(a_s)\leq f_w(a_w)$. Clearly,
	$f_i(b)\leq f_w(b)\leq f_w(b'')$ and $f_j(b')\leq f_w(b')\leq f_w(b'')$.
	
	From (1)--(3) it follows that $g$ is an $EMV$-morphism.  Clearly, $f\approx g$. Indeed, for each $b\in K$, there is $i\in I$ such that
	$b\leq a_i$ for all $x\in [0,b]$ we have
	$g_{i,b}(x)=f_i(x)$. That is, $f\approx g$.
\end{proof}

Consider the assumptions of Proposition \ref{9.3}. In the next theorem, we show that if, for each $b\in K$, we put
a special function $g_b$ to be equal to  $f_j|_{_{[0,b]}}$ from $\{f_i|_{_{[0,b]}}\mid i\in I, b\leq a_i\}$, then $\{g_b\mid b\in K\}$ is again an $EMV$-morphism
which is similar to $f$. For each $i\in I$, there exists $b\in K$ such that $a_i\leq b$. Also, there is $j\in I$ such that
$a_i,b\leq a_j$ and  $f_i(a_i)\leq f_j(a_j)$. Since $f$ is an $EMV$-morphism, then $f_i(x)=f_j(x)\wedge f_i(a_i)$ for
all $x\in [0,a_i]$, and so $f_j(b)\geq f_j(a_i)\geq f_i(a_i)$. Set $g_b:=f_j|_{_{[0,b]}}$. Then $g_b(b)\geq f_i(a_i)$.
We can easily check that $\{g_b(b)\mid b\in K\}$ has the following property:
\begin{equation}\label{additin pro}
\forall t\in I\ \exists\, b\in K\ : a_t\leq b \ \& \ f_t(a_t)\leq g_b(b).
\end{equation}
Whence, $\{g_b(b)\mid b\in K\}$ is a full subset of $M_2$.

\begin{prop}\label{9.3.1}
Let $K$ be a full subset of idempotents of $M_1$ and $f=\{f_i\mid i\in I\}:M_1\ra M_2$ be an $EMV$-morphism with $\mathsf{e}(f)=\{a_i\mid i\in I\}$.
	If, for each $b\in K$, $g_{b}:=f_{i}|_{_{[0,b]}}$, where $i\in I$ and $b\leq a_i$ such that $\{g_b(b)\mid b\in K\}$ satisfies the condition {\rm(\ref{additin pro})},
	then $g=\{g_b\mid b\in K\}$ is an $EMV$-morphism and $g\approx f$.
\end{prop}

\begin{proof}
	Since $\mathsf{e}(f)$ is a full subset of $M_1$, then $\mathsf{e}(g)=K$.
	
	(1) Since $\{f_i(a_i)\mid i\in I\}$ is a full subset of $M_2$, then by (\ref{additin pro}), $\{g_b(b)\mid b\in K\}$ is a
	full subset of $M_2$, too.
	
	(2) If $b,b'\in K$, then there exist $i,j\in I$ such that $b\leq a_i$, $b'\leq a_j$, $g_b(b)=f_i(b)$ and $g_{b'}(b')=f_j(b')$.
	Put $t\in I$ such that $a_i,a_j\leq a_t$ and
	$f_i(a_i),f_j(a_j)\leq f_t(a_t)$. By  (\ref{additin pro}), there exists $b''\in K$ with $f_t(a_t)\leq g_{b''}(b'')$.
	It follows that $g_b(b),g_{b'}(b')\leq g_{b''}(b'')$.
	
	(3) Let $b,b'\in K$ such that $g_b(b)\leq g_{b'}(b')$. Then there are $i,j\in I$ such that $b\leq a_i$, $b'\leq a_j$,
	$g_b=f_i|_{_{[0,b]}}=g_{i,b}$ and $g_{b'}=f_j|_{_{[0,b']}}=g_{j,b'}$. By Proposition {\rm \ref{9.3}}, for all
	$x\in [0,b\wedge b']$, we have $g_b(x)=g_{b'}(x)\wedge g_b(b)$.
	From (1)--(3) it follows that $\{g_b\mid b\in K\}$ is an $EMV$-morphism.
Now, we show that $f\approx g$. Let $b\in K$. Then there is $i\in I$ such that $b\leq a_i$ and $g_b=f_i|_{_{[0,b]}}$ and so
\[g_b(b)\leq f_i(b)\leq f_i(a_i)\quad \& \quad \forall x\in [0,b],\  g_b(x)=f_i(x)\wedge g_b(b).\]
Therefore, $g=\{g_b\mid b\in K\}\approx f$.
\end{proof}

\begin{rmk}\label{I(M)}
	(i) In the last proposition, the definition of $g=\{g_b\mid b\in K\}$ does not depend on the choice of $f_i$ for given $b$.
	Indeed, if $h=\{h_b\mid b\in K\}$ is defined by $h_b=f_j(b)$, where $b\leq a_j$ and $\{h_b(b)\mid b\in K\}$
	satisfies in (\ref{additin pro}), then
	by Proposition \ref{9.3.1}, $h\approx f\approx g$. So we can feel free for choosing a suitable $f_i$ related to $b$. The only thing we must care is
	(\ref{additin pro}).
	
	(ii) If we consider $K=\mI(M_1)$ in Proposition \ref{9.3.1}, then we have $\mathsf{e}(g)=\mathsf{e}(h)=\mI(M_1)$ and $g\approx h\approx f$.
	Thus without loss of generality we can assume that $\mathsf{e}(f)=\mI(M_1)$.
\end{rmk}

An $EMV$-morphism $f=\{f_i\mid i\in I\}:M_1\ra M_2$ with $\mathsf{e}(f)=\{a_i\mid i\in I\}$ is called an {\em $EMV$-$\approx$-isomorphism}
if there exists an $EMV$-morphism $g=\{g_j\mid j\in J\}:M_2\ra M_1$ such that $g\circ f\approx Id_{M_1}$ and $f\circ g\approx Id_{M_2}$, and we say that $M_1$ and $M_2$ are $\approx$-{\it isomorphic}.

\begin{prop}\label{equal to identity}
	Let $f=\{f_i\mid i\in I\}:M\ra M$ be an $EMV$-morphism such that $\mathsf{e}(f)=\{a_i\mid i\in I\}$.
	Then $f\approx Id_M$ if and only if, for all  $i\in I$, we have
	\begin{itemize}[nolistsep]
		\item[{\rm (i)}] $f_i(x)\leq x$ for all $x\in [0,a_i]$;
		\item[{\rm (ii)}] $f_i(x)=x$ for all $x\leq f_i(a_i)$.
	\end{itemize}

\begin{proof}
	Let $f\approx Id_M$ and $i$ be an arbitrary element of $I$. Then there is $b\in\mI(M)$ such that $a_i\leq b$ and
	$f_i(x)=Id_M(x)\wedge f_i(a_i)$ for all $x\in [0,a_i]$. Clearly, for all $x\in [0,a_i]$, we have $f_i(x)\leq x$, specially,
	$f_i(a_i)\in [0,a_i]$. Also, for each $x\leq f_i(a_i)$,
	$f_i(x)=x\wedge f_i(a_i)=x$. Conversely, let (i) and (ii) hold. Put $i\in I$. Then by (i), $f_i(a_i)\leq a_i$. Since
	$\mathsf{e}(f)$ is a full subset of $M$, then there exists $j\in I$ such that $a_i\leq f_j(a_j)$ and so by definition,
	for all $x\in [0,a_i]$, we have 	$f_i(x)=f_j(x)\wedge f_i(a_i)$. Since $a\leq f_j(a_j)$, by (ii),
	$f_j(x)=x$ for all $x\in [0,a_i]$, and thus
	\[f_i(x)=f_j(x)\wedge f_i(a_i)=x\wedge f_i(a_i)=Id_M(x)\wedge f_i(a_i),\quad\forall x\in [0,a_i].\]
	Therefore, $f\approx Id_M$.
\end{proof}

\begin{cor}\label{isomorphism}
	Let $f=\{f_i\mid i\in I\}:M_1\ra M_2$ be an $EMV$-$\approx$-isomorphism with $\mathsf{e}(f)=\{a_i\mid i\in I\}$, and let $g=\{g_j\mid j\in J\}:M_2\ra M_1$ with $\mathsf{e}(g)=\{b_j\mid j\in J\}$ be an $EMV$-morphism such that $f\circ g\approx Id_{M_2}$ and $g\circ f\approx Id_{M_1}$. Then $f\circ g=\{f_i\circ g_j\mid (i,j)\in U\}$ and $g\circ f=\{g_j\circ f_i\mid (j,i)\in V\}$, where $U=\{(i,j)\in I\times J\mid f_i(a_i)\leq b_j\}$ and	 $V=\{(s,t)\in J\times I\mid g_s(b_s)\leq a_t\}$.
	 Consider the $EMV$-morphisms $h=\{h_{i,j}\mid (i,j)\in U\}:M_1\ra M_1$ and $k=\{k_{s,t}\mid (s,t)\in V\}:M_2\ra M_2$ defined by
	 $h_{i,j}(x)=g_j(f_i(x))$ if $(i,j)\in U$ and $x\leq g_j(f_i(a_i))$, and
	 $k_{s,t}(x)=f_t(g_s(x))$ if $(s,t)\in V$ and $x\leq f_s(g_t(b_t))$. 	
	 By Proposition {\rm \ref{equal to identity}}, for all $(i,j)\in U$ and
	 $(s,t)\in V$, we get
	 \[h_{i,j}(y)=f_j\circ g_i(y)=y,\  \forall y\leq f_j(g_i(b_i)),  \quad \quad k_{s,t}(x)= g_t\circ f_s(x)=x,\  \forall x\leq g_t(f_s(a_s)).\]
	 Therefore, $h\approx Id_{M_1}$ and $k\approx Id_{M_2}$.
\end{cor}

\begin{exm}
	In Example \ref{exm3.4}, if $\{f_i:M_i\ra N_i\mid i\in\mathbb{N}\}$ is a family of $MV$-isomorphisms, then we can easily see that
	$f=\{f_I\mid I\in S\}$ is an $EMV$-$\approx$-isomorphism.
\end{exm}

\end{prop}
\begin{defn}\label{def100}
	Let $f=\{f_i\mid i\in I\}:M_1\ra M_2$ and $g=\{g_j\mid j\in J\}:M_1\ra M_2$  be two $EMV$-morphisms
	such that $\textsf{e}(f)=\{a_i\mid i\in I\}$ and $\textsf{e}(g)=\{b_j\mid j\in J\}$.
     Let $x,y\in M_1$. We say that $f(x)=g(y)$ if, for each $i\in I$ and each $j\in J$ such that
	$x\leq a_i$, $y\leq b_j$ and $f_i(a_i)\leq g_j(b_j)$, then $f_i(x)=g_j(y)\wedge f_i(a_i)$.
\end{defn}
	Consider the assumptions in Definition \ref{def100}. Then $f(x)=g(y)$ if and only if, for each	$i\in I$ and each $j\in J$ such that
	$x\leq a_i$, $y\leq b_j$ and $f_i(a_i)\geq g_j(b_j)$, we have $g_j(y)=f_i(x)\wedge g_j(b_j)$.  Indeed,
	let $i\in I$ and $j\in J$ such that $x\leq a_i$, $y\leq b_j$ and $g_j(b_j)\leq f_i(a_i)$. Then we can find $k\in J$
	such that $b_j\leq b_k$ and $f_i(a_i)\leq g_k(b_k)$. Since $f(x)=g(y)$, then by definition,
	$f_i(x)=g_k(y)\wedge f_i(a_i)$ and
	\[f_i(x)\wedge g_j(b_j)=g_k(y)\wedge f_i(a_i)\wedge g_j(b_j)=g_k(x)\wedge g_j(b_j)=g_j(y).\]
	The proof of the converse is similar.

In the next proposition, we find a relation between $\approx$ and the concept introduced in Definition  \ref{def100}.

	\begin{prop}\label{pr:f=g}
		Let $f=\{f_i\mid i\in I\}:M_1\ra M_2$ and $g=\{g_j \mid j\in J\}:M_1\ra M_2$  be two $EMV$-morphisms
		such that $\textsf{e}(f)=\{a_i\mid i\in I\}$ and $\textsf{e}(g)=\{b_j\mid j\in J\}$.
		\begin{itemize}[nolistsep]
			\item[{\rm (i)}]  $f\approx g$ \iff $f(x)=g(x)$ for all $x\in M_1$.
			\item[{\rm (ii)}] $f(x)=f(y)$ if and only if, for each $i\in I$ with $x,y\leq a_i$, we have $f_i(x)=f_i(y)$.
		\end{itemize}
	\end{prop}
	
	\begin{proof}
		(i)  The proof is straightforward.
		
		(ii) Suppose that $i\in I$ such that $x,y\leq a_i$. Then $f_i(x)=f_i(y)\wedge f_i(a_i)=f_i(y)$.
		Conversely, let $x\leq a_i$ and $y\leq a_j$ such that $f_i(a_i)\leq f_j(a_j)$ for some $i,j\in I$.
		Then there exists $k\in I$ such that $a_i,a_j\leq a_k$ and $f_i(a_i),f_j(a_j)\leq f_k(a_k)$.
		By the assumption, $f_k(x)=f_k(y)$, so that
		$f_j(y)\wedge f_i(a_i)=(f_k(y)\wedge f_j(a_j))\wedge f_i(a_i)=f_k(y)\wedge f_i(a_i)=f_k(x)\wedge f_i(a_i)=f_i(x)$.
		That is, $f(y)=f(x)$.
	\end{proof}

\begin{prop}\label{9.4}
	Let $f=\{f_i\mid i\in I\}:M_1\ra M_2$ with $\mathsf{e}(f)=\{a_i\mid i \in I\}$ be an $EMV$-morphism. Then
	\begin{eqnarray*}
		\ker(f):=\{(x,y)\in M_1\times M_1\mid  f_i(x)=f_i(y),\ \forall a_i\geq x,y \}
	\end{eqnarray*}
	is a congruence relation on $M_1$.
\end{prop}

\begin{proof}
	Clearly, $\ker(f)$ is reflexive and symmetric. Let $(x,y),(y,z)\in\ker(f)$. Put $i\in I$ such that
	$x,y,z\leq a_i$. Then $f_i(x)=f_i(y)=f_i(z)$. For each $j\in I$ if $x,z\leq a_j$, then there exists $s\in I$ such that
	$a_j,a_i\leq a_s$ and $f_i(a_i),f_j(a_j)\leq f_s(a_s)$. It follows that
	\[f_j(x)=f_s(x)\wedge f_j(a_j)=f_s(y)\wedge f_j(a_j)=f_s(z)\wedge f_j(a_j)=f_j(z).\]
	Therefore, $\ker(f)$ is transitive.
	Now, let $(x,y)\in\ker(f)$ and $x,y\leq a_i$ for  $i\in I$. Given $u$, if $w\in I$ such that $x\vee u,y\vee u\leq a_w$, then
	we have $f_w(x)=f_w(y)$ and so $f_w(x\vee u)=f_w(x)\vee f_w(u)=f_w(y)\vee f_w(u)=f_w(x\vee u)$ (since $f_w$ is an
	$MV$-homomorphism) consequently, $(x\vee u,y\vee u)\in\ker(f)$. In a similar way, $(x\oplus u,y\oplus u)\in\ker(f)$.
	Also, for each $k\in I$ with $x\wedge u,y\wedge u\leq a_k$, there is $z\in I$ such that $a_z\geq a_k,a_i,u$ and
	$f_z(a_z)\geq f_i(a_i),f_k(a_k)$, whence
	$f_z(x)=f_z(y)$ and $f_k(c)=f_z(c)\wedge f_k(a_k)$ for all $c\in [0,a_k]$. Thus, from
	\[ f_k(x\wedge u)=f_z(x\wedge u)\wedge f_k(a_k)=f_z(x)\wedge f_z(u)\wedge f_k(a_k)=f_z(y)\wedge f_z(u)\wedge f_k(a_k)=
	f_z(y\wedge u)\wedge f_k(a_k)=f_k(y\wedge u),\]
	we conclude that $(x\wedge u,y\wedge u)\in\ker(f)$. By Proposition \ref{3.9}, it remains to show that for each $(x,y)\in\ker(f)$,
	there is $b\in \mI(M_1)$ such that $x,y\leq b$ and $(\lam_b(x),\lam_b(y))\in \ker(f)$.
	Let $(x,y)\in\ker(f)$. Then there exists $i\in I$ such that $x,y\leq a_i$ and so $f_i(x)=f_i(y)$. Set $b:=a_i$. For all $j\in I$
	with $\lam_b(x),\lam_b(y)\leq a_j$, there is $k\in I$ such that $a_k\geq b,a_j$ and
	$f_k(a_k)\geq f_i(a_i),f_j(a_j)$. It follows that
	\begin{equation}
	\label{9.4.R1} f_j(c)=f_k(c)\wedge f_j(a_j),\quad \forall c\in [0,a_j].
	\end{equation}
	Consequently,
	\begin{eqnarray*}
	f_j(\lam_b(x))&=& f_k(\lam_b(x))\wedge f_j(a_j)=f_k(\lam_{a_k}(x)\wedge b)\wedge f_j(a_j), \mbox{ by Proposition \ref{3.5}} \\
	&=& \lam_{f_k(a_k)}(f_k(x))\wedge f_k(b)\wedge f_j(a_j), \mbox{ since $f_k:[0,a_k]\ra [0,f_k(a_k)]$ is an $MV$-homomorphism}\\
	&=& \lam_{f_k(a_k)}(f_k(y))\wedge f_k(b)\wedge f_j(a_j)=f_k(\lam_{a_k}(y)\wedge b)\wedge f_j(a_j)\\
	&=& f_k(\lam_b(y))\wedge f_j(a_j)=f_j(\lam_b(y)).
    \end{eqnarray*}
    Therefore, $\ker(f)$ is a congruence relation on $M_1$.
\end{proof}

\begin{cor}\label{9.5}
	If $\theta$ is a congruence relation on an $EMV$-algebra $(M;\vee,\wedge,\oplus,0)$, then the natural map
	$\pi_\theta=\{\pi_a\mid a\in \mI(M)\}:M\ra M/\theta$, where $\pi_a(x)=x/\theta$ for all $x\in [0,a]$ and for all $a\in\mI(M)$, is an $EMV$-morphism.
\end{cor}

\begin{proof}
	The proof is straightforward.
\end{proof}

If $\mathcal{EMV}$ is the set of $EMV$-algebras, $\mathcal{EMV}$ is a variety with respect to $EMV$-homomorphisms, see \cite[Thm 3.11]{Dvz}. In below we will study $\mathcal{EMV}$ with respect to $EMV$-morphisms and therefore, it will be not more a variety in such a point of view.

Let $\approx$ be the equivalence given by Definition \ref{hhhh}. We say that an $EMV$-morphism $g:M\to N$ which satisfies some conditions is $\approx$-{\it unique} if $h:M\to N$ is another $EMV$-morphisms which satisfies the same conditions, then $g$ and $h$ are similar, i.e. $g\approx h$.

In the following result, first we show that in the class $\mathcal{EMV}$ there is a weak form of the categorical product of a family of $EMV$-algebras, called the $\approx$-product: Let $\{M_i\mid i \in I\}$ be a family of $EMV$-algebras. We say that an $EMV$-algebra $\approx \prod_{i\in I}M_i$ together with a  family of $EMV$-morphisms $\{\pi_i\mid i\in I\}$, where $\pi_i:\approx \prod_{i\in I} M_i \to M_i$ is an $EMV$-morphism for each $i \in I$, is a $\approx$-{\it product} of  $\{M_i\mid i \in I\}$, if $M$ is another $EMV$-algebra together with a family
$\{f_i\mid i \in I\}$ of $EMV$-morphisms $f_i:M \to M_i$ for each $i \in
I$, then there is an $\approx$-unique $EMV$-morphism $g:M \to \approx \prod_{i\in I}M_i$ such that $\pi_i\circ g \approx f_i$ for each $i\in I$.
Using this, finally we show that the category $\mathbb{EMV}$ admits a product in the pure categorical sense.

\begin{thm}\label{9.6}
	The category $\mathbb{EMV}$ is closed under product.
\end{thm}

\begin{proof}
We start with the class $\mathcal{EMV}$ of $EMV$-algebras.
Let $\{(M_i;\vee,\wedge,\oplus,0)\mid i\in I\}$ be a family of	 $EMV$-algebras. Consider the $EMV$-algebra $\prod_{i\in I} M_i$, which is	 the direct product of $EMV$-algebras. In what follows, we show that that	 this $\prod_{i\in I}M_i$ together with projections is the $\approx$-product of $\{M_i\mid i \in M\}$ in the class $\mathcal{EMV}$.

By Remark \ref{9.1.1}, for each $i\in I$, the map $\pi_{i}:\prod_{i\in I}
M_i\ra M_i$ is an $EMV$-morphism. Let $(M;\vee,\wedge,\oplus,0)$ be an
$EMV$-algebra and $f_i:M\ra M_i$ be an $EMV$-morphism for each $i\in I$. According to Remark \ref{I(M)}, we can suppose
that 
$f_i=\{f_{i,a}\mid a\in \mathcal I(M)\}$
and	 $\textsf{e}(f_i)=\mI(M)$ for all $i\in I$. Define $g=\{g_a\mid a\in \mI(M)\}:M\ra \prod_{i\in I} M_i$, where $g_a:[0,a]\ra [0,(f_{i,a}(x))_{i\in I}]$ sending $x$ to $(f_{i,a}(x))_{i\in I}$ for all $x\in [0,a]$. Clearly, $g$	is an $EMV$-morphism. On the other hand,	$\pi_{i}:\prod_{i\in I} M_i\ra M_i$ is a strong $EMV$-homomorphism, which follows from Remark \ref{9.1.1}, and the mapping	$\pi=\{\pi_{i,(a_i)_{i\in I}}\mid a_i\in \mI(M_i) \forall i\in I\}$ is an $EMV$-morphism, where	$\pi_{i,(a_i)_{i\in I}}: [0,(a_i)_{i\in I}]\ra [0,a_i]$ is the $i^{th}$ natural projection map.	 We claim that $\pi_i \circ g\approx f_i$ for all $i\in I$. Choose $i\in I$ and set
	\begin{eqnarray*}
	 U:=\{(a,(b_i)_{i\in I})\mid a\in \mI(M), (b_i)_{i\in I}\in \mI(\prod_{i\in I}M_i), g_a(a)\leq (b_i)_{i\in I}\}.
	\end{eqnarray*}
    Then we have
    \begin{equation}\label{9.6R1}
    U=\{(a,(b_i)_{i\in I})\mid a\in \mI(M), (b_i)_{i\in I}\in \mI(\prod_{i\in I}M_i), (f_{i,a}(a)\leq b_i,\forall i\in I)\}.
    \end{equation}
    By Proposition \ref{9.2}, we know that
    \[\pi_i \circ g=\{\pi_{i,(b_i)_{i\in I}}\circ g_a\mid (a,(b_i)_{i\in I})\in U\}\quad \& \quad Dom(\pi_{i,(b_i)_{i\in I}}\circ g_a)=[0,a]. \]
    We will show that
    \begin{equation}\label{9.6R2}
    \forall c\in \mI(M)\ \exists\, (a,(b_i)_{i\in I})\in U: c\leq a \quad \& \quad f_{i,c}=(\pi_{i,(b_i)_{i\in I}}\circ g_a)\wedge f_{i,c}(c).
    \end{equation}
Take $c\in\mI(M)$. For each $i\in I$, we set  $b_i=f_{i,c}(c)$ and $a=c$. Then $(a,(b_i)_{i\in I})\in U$ and $c\leq a$.
    For all $x\in [0,c]$,
    \[ (\pi_{i,(b_i)_{i\in I}}\circ g_a)(x)\wedge f_{i,c}(c)=f_{i,a}(x)\wedge f_{i,c}(x)=f_{i,c}(x). \]
    That is, $\pi_i \circ g\approx f_i$.
    Now, we show that $g$ is $\approx$-unique. Let $\{c_t\mid t\in K\}\s \mI(M)$,
    $h_t:[0,c_t]\ra [0,h_t(c_t)]$ be an $MV$-homomorphism for all $t\in K$, and let
    $h=\{h_t\mid t\in K\}:M\ra \prod_{i\in I}M_i$ be an $EMV$-morphism such that
    $\pi_{i}\circ h\approx f_i$ for all $i\in I$.
    We claim that $h\approx g$.
    It is enough to show that, for all $a\in \mI(M)$, there exists $t\in K$ such that $a\leq c_t$ and
    \[\forall x\in [0,a]\quad (f_{i,a}(x))_{i\in I}=g_a(x)=h_t(x)\wedge g_a(a)=h_t(x)\wedge (f_{i,a}(a))_{i\in I}.\]
    Put $a\in \mI(M)$ and let
    \begin{equation}\label{9.6R3}
    V=\{(t,(b_i)_{i\in I})\mid t\in K, (b_i)_{i\in I}\in \mI(\prod_{i\in I}M_i), h_t(c_t)\leq (b_i)_{i\in I}\}.
    \end{equation}
    Since $h$ is an $EMV$-morphism, there are $t_1,t_2\in K$ such that $a\leq c_{t_1}$ and $g_a(a)\leq h_{t_2}(c_{t_2})$.
    Put $t\in K$ such that $c_t\geq c_{t_1},c_{t_2}$ and $h_t(c_t)\geq h_{t_1}(c_{t_1}),h_{t_2}(c_{t_2})$, which implies that
    $h_t(c_t)\geq g_a(a)$ and $c_t\geq a$.
    Clearly, $h_t(c_t)\in \mI(\prod_{i\in I} M_i)$ (since $a\in [0,c_t]$ is an idempotent and $h_t:[0,c_t]\ra [0,h_t(c_t)]$ is an
    $MV$-homomorphism). By definition of $V$, we get that $(c_t,h_t(c_t))\in V$. Since $h_t(c_t)\geq g_a(a)=(f_{i,a}(a))_{i\in I}$, by
    Proposition \ref{eq.property}(i), from $f_i\approx \pi_i \circ h$ for all $i\in I$, it follows that
    \[ f_{i,a}(x)=(\pi_{i,h_t(a)}\circ h_t)(x) \wedge f_{i,a}(a),\quad \forall x\in [0,a],\ \forall i\in I.  \]
    Hence, $(f_{i,a}(x))_{i\in I}=h_t(x)\wedge (f_{i,a}(a))_{i\in I}$ for all $x\in [0,a]$. Therefore,
    $g_a(x)=h_t(x)\wedge g_a(a)$ for all $x\in [0,a]$. That is, $g\approx h$, see Proposition \ref{pr:f=g}. Whence, $\prod_{i\in I}M_i$ together with $\{\pi_i\mid i \in I\}$ is the $\approx$-product of the family $\{M_i\mid i \in I\}$ of $EMV$-algebras.

Now, consider the category $\mathbb{EMV}$. If for all $i\in I$,
$[f_i]:M\ra M_i$ is a morphism in the category $\mathbb{EMV}$, then by the above results
    there exists a  $\approx$-unique $EMV$-morphism $g:M\ra \prod_{i\in
I}M_i$ such that $\pi_i\circ g=f_i$, for all $i\in I$ and so
    $[\pi_i]\circ [g]=[f_i]$ for all $i\in I$. In addition, if $[h]=:M\ra
\prod_{i\in I}M_i$ is another morphism of $\mathbb{EMV}$ such that
    $[\pi_{i}]\circ [h]=[f_i]$ for all $i\in I$, then $\pi_{i}\circ
h\approx f_i$ for all $i\in I$. Thus by the above results we have
    $h\approx g$, that is $[h]=[g]$. Therefore, $\mathbb{EMV}$ is closed under products.
\end{proof}

\section{Free $EMV$-algebras and Weakly Free $EMV$-algebras}

It is well known that free objects can be defined in a concrete category. In particular, in the variety $\mathcal{EMV}$ of $EMV$-algebras with $EMV$-homomorphisms, free $EMV$-algebras exist for each set $X$. However, comparing of such free $EMV$-algebras with free $MV$-algebras is complicated. Therefore, we can say more using $EMV$-morphisms for definition of free objects. This is possible because, according to the results which were obtained in the previous sections, $\mathbb{EMV}$ is very similar to a
concrete category. It enables us to introduce an object which has many similarities with the free objects. In the section, we show that $EMV$-morphisms enable us to study also free $EMV$-algebras on a set $X$. We show that if $X$ is a finite set, then the free $MV$-algebra on $X$ is also a free $EMV$-algebra. To show how it is with infinite $X$, we introduce a so-called weakly free $EMV$-algebra. Then the free $MV$-algebra on $X$ is a weakly free $EMV$-algebra on $X$.

\begin{defn}\label{de:free}
Let $X$ be a set and $\tau:X\ra A$ be a map from $X$ to an $EMV$-algebra $A$.
We say that an $EMV$-algebra $A$ is a {\it free $EMV$-algebra} on $X$ if, for each $EMV$-algebra $M$ and for each map $f$ from $X$ into $M$, there is an $\approx$-unique $EMV$-morphism $\varphi =\{\varphi_i \mid i\in I\}:A\ra M$
with $\mathsf{e}(\varphi)=\{a_i\mid i\in I\}\subseteq \mathcal I(A)$ such that $\varphi\circ \tau=f$, which means that, for each $x\in X$ and each $i\in I$, $f(x)=\varphi_i(\tau(x))$. As usual, we denote the free $EMV$-algebra $A$ on $X$ by $A=F(X)$, and the set $X$ is the set of ``generators" of $F(X)$, see Lemma \ref{le:F(X)}.
\end{defn}

The basic properties of the free $EMV$-algebras are as follows.

\begin{prop}\label{pr:free}
Let $F(X)$ be a free $EMV$-algebra on a set $X$ and with a mapping $\tau$.
\begin{itemize}[nolistsep]
\item[{\rm (1)}] Let $M$ be an $EMV$-algebra and $f:X\to M$. If $\phi=\{\phi_i\mid i \in I\}$ with $\mathsf{e}=\{a_i\mid i \in I\}$ is an $EMV$-morphism from $F(X)$ into $M$, then $\tau(x)\le a_i$ for each $x \in X$ and each $i \in I$.

\item[{\rm (2)}] The mapping $\tau$ is injective.

\item[{\rm (3)}] If $F'(X)$ is another free $EMV$-algebra on $X$, then $F(X)$ and $F'(X)$ are $\approx$-isomorphic.

\item[{\rm (4)}] If $|X|=|X'|$, then $F(X)$ and $F(X')$ are $\approx$-isomorphic.
\end{itemize}
\end{prop}

\begin{proof}

(1) Since every $\phi_i:[0,a_i]\to [0,\phi_i(a_i)]$ is an $MV$-homomorphism such that $\phi_i(\tau(x))=f(x)$, we see that $\tau(x)\le a_i$.

(2) Assume that there are $x_1,x_2\in X$ such that $x_1\ne x_2$ and $\tau(x_1)=\tau(x_2)$. Let $M$ be an $EMV$-algebra with at least two elements and let $f:X \to M$ be a mapping such that $f(x_1)\ne f(x_2)$. Then there is an $\approx$-unique $EMV$-morphism $\phi=\{\phi_i\mid i \in I\}$ with $\mathsf{e}=\{a_i\mid i \in I\}$ such that $\phi\circ \tau =f$, that is
$\phi_i(\tau(x))=f(x)$, $x \in X$, $i \in I$. By (1), every $\tau(x)\le a_i$, so we have $f(x_1)=\phi_i(\tau(x_1))=\phi_i(\tau(x_2))=f(x_2)$ which contradicts $f(x_1)\ne f(x_2)$.

(3) Let $\tau:X \to F(X)$ and $\tau':X \to F'(X)$ be given. Then there are two $\approx$-unique $EMV$-morphisms $\phi: F(X)\to F'(X)$ and $\psi:F'(X)\to F(X)$ such that $\phi \circ \tau = \tau'$ and $\psi \circ \tau'= \tau$. Then $\phi\circ \psi \circ \tau=\tau'$ and $\psi \circ \phi \circ \tau = \tau$ which gives $\phi \circ \psi \approx Id_{F'(X)}$ and $\psi\circ \phi \approx Id_{F(X)}$, in other words $F(X)$ and $F'(X)$ are $\approx$-isomorphic.

(4) Let $\tau:X \to F(X)$ and $\tau':X'\to F(X')$ be injective mappings determining $F(X)$ and $F(X')$, respectively. Since $|X|=|X'|$, there is a bijective mapping $\beta:X\to X'$.
There are an $\approx$-unique $EMV$-morphism $\phi=\{\phi_i\mid i \in I\}:F(X)\to F(X')$ with $\mathsf e(\phi)=\{a_i\mid i\in I\}$ and an $\approx$-unique $EMV$-morphism $\psi=\{\psi_j\mid j \in J\}:F(X')\to F(X)$ with $\mathsf e(\psi)=\{b_j\mid j \in J\}$ such that  $\tau'\circ \beta=\phi \circ \tau$ and $\tau\circ \beta^{-1} = \psi \circ \tau'$.
Then $(\psi \circ \phi)\circ \tau = \psi \circ \tau' \circ \beta = \tau \circ \beta^{-1}\circ \beta = \tau$ which yields $\psi\circ \phi \approx Id_{F(X)}$. In the same way we prove $\phi \circ \psi \approx Id_{F(X')}$. Therefore, $F(X)$ and $F(X')$ are $\approx$-isomorphic.
\end{proof}

Consider the class $\mathcal{EMV}$ of $EMV$-algebras. To show a relationship between $F(X)$ and $X$, we introduce a stronger type of $EMV$-subalgebras, full subalgebras, which will entail that between $\tau(X)$ and $F(X)$, there is no other full subalgebra as $F(X)$, so we can speak that $X$ is a generator of $F(X)$.

 \begin{defn}\label{subalg}
A non-empty subset $X$ of an $EMV$-algebra $M$ is said to be a {\em full subalgebra} of $M$ if $X$ is closed under $\oplus$, $\vee$, $\wedge$ and $0$ and satisfies 	the following conditions:
 \vspace{2mm}
 \begin{itemize}[nolistsep]	
 	\item[{\rm (i)}] $X\cap\mI(M)$ is a full subset of $M$;
 	
 	\item[{\rm (ii)}]  for each $a\in X\cap\mI(M)$, $([0,a]\cap X;\oplus,\lam_a,0,a)$ is an $MV$-subalgebra of $([0,a];\oplus,\lam_a,0,a)$.
 \end{itemize}
 \end{defn}

Clearly, any full subalgebra of an $EMV$-algebra $M$ is an $EMV$-subalgebra of $M$. The converse is not true, in general. Indeed, the set $\{0\}$ is an $EMV$-subalgebra of any $EMV$-algebra $M$, but it is not a full subalgebra of $M$ whenever $M\ne \{0\}$. Moreover, if $X$ is a full subalgebra of an $EMV$-algebra $M$, then the inclusion map $i:X\ra M$ is an $EMV$-morphism. Moreover, the image of any $EMV$-morphism $f:M_1\ra M_2$ is a full subalgebra of $M_2$.

\begin{lem}\label{le:F(X)}
    Let $F(X)$ be a free $EMV$-algebra on a set $X$ with an embedding $\tau$. If $A$ is a full
subalgebra of $F(X)$ containing $\tau(X)$, then $A$ is equal to $F(X)$.
\end{lem}

\begin{proof}
Consider the diagram given by Figure 1.

Let $\tau: X\to F(X)$ and $\tau': X \to A$ be such mappings that $\tau(x)=\tau'(x)$ for each $x \in X$. Let $\mu:A \to F(X)$ be the natural embedding, i.e. $\mu(y)=y$ for each $y \in A$.
Since $A$ is a full subalgebra of $F(X)$, the inclusion map $\mu:A\ra
F(X)$ can be viewed as an $EMV$-morphism $\mu =\{\mu_a\mid a \in \mathcal I(A)\}$, where $\mu_a=\mu|_{[0,a]}$ for each $a\in \mathcal I(A)$.
The definition of $F(X)$ entails that there is an $\approx$-unique $EMV$-morphism $\phi=\{\phi_i\mid i \in I\}: F(X)\to A$ with $\mathsf e(\phi)=\{a_i \mid i \in I\}$ such that $\phi \circ \tau = \tau'$.

On the other hand, there is an $\approx$-unique $EMV$-homomorphism $\psi=\{\psi_j\mid j \in J\}: F(X)\to F(X)$ with $\mathsf e(\psi) =\{b_j \mid j \in J\}$ such that $\psi \circ \tau = \tau$. Since $\tau = \mu\circ \tau' = \mu\circ \phi \circ \tau$, we have $\mu\circ \phi \approx \psi \approx Id_{F(X)}$.
Set $U:=\{(i,a)\in I\times \mI(A)\mid \phi_i(a_i)\leq a\}$. By
definition, $\mu\circ\phi=\{\mu_a\circ \phi_i\mid (i,a)\in U\}$.
Let $v$ be an arbitrary element of $F(X)$. Then there exists $i\in
I$ such that $v\leq a_i$ and $v\leq \phi_i(a_i)$.  Since $\mu\circ\phi\approx Id_{F(X)}$, by Proposition \ref{equal to identity}, for each $a\in\mI(A)$ with $\phi_i(a_i)\leq a$, we have $\mu_a\circ \phi_i(v)=v$. That is, $v=\phi_i(v)\in A$. Therefore, $A=F(X)$.

\setlength{\unitlength}{1mm}
\begin{figure}[!ht]
	\begin{center}
		\begin{picture}(40,20)
		\put(10,17){\vector(2,0){14}}
        \put(8,14){\vector(0,-1){12}}
\put(10,-1){\vector(2,0){14}}
\put(27,15){\vector(0,-2){13}}
\put(10,15){\vector(1,-1){14}}
        \put(5,-2){\makebox(4,2){{ $A$}}}
		\put(5,16){\makebox(4,2){{ $X$}}}
		\put(27,16){\makebox(4,2){{ $F(X)$}}}
		\put(15,19){\makebox(4,2){{ $\tau$}}}
		\put(1,7){\makebox(4,2){{$\tau'$}}}
\put(28,7){\makebox(4,2){{$\psi$}}}
\put(22,9){\makebox(4,2){{$\phi$}}}
\put(10,9){\makebox(4,2){{$\tau$}}}
\put(16,0){\makebox(4,2){{$\mu$}}}
		\put(25,14){\vector(-1,-1){13}}
        \put(27,-2){\makebox(4,2){{ $F(X)$}}}
		\end{picture}
		\caption{ }
	\end{center}
\end{figure}
\end{proof}

Therefore, we can say that $X$ is a generator of $F(X)$.


In the next theorem we show that the free $EMV$-algebra and the free $MV$-algebra on a finite set $X$
coincide.

\begin{thm}\label{9.7}
	Let $X=\{x_1,x_2,\ldots ,x_n\}$ be an arbitrary finite set and $F(X)$ be the free $MV$-algebra on $X$. Then $F(X)$ is the free
	$EMV$-algebra on $X$.
\end{thm}

\begin{proof}
Let $M$ be an $EMV$-algebra and $f:X\ra M$ be a map.
Consider the diagram given by Figure 2.

\setlength{\unitlength}{1mm}
\begin{figure}[!ht]
	\begin{center}
		\begin{picture}(40,20)
		\put(12,17){\vector(2,0){12}} \put(8,14){\vector(0,-1){12}}
		\put(5,-2){\makebox(4,2){{ $M$}}}
		\put(5,16){\makebox(4,2){{ $X$}}}
		\put(27,16){\makebox(4,2){{ $F(X)$}}}
		\put(15,19){\makebox(4,2){{ $\tau$}}}
		\put(1,7){\makebox(4,2){{$f$}}}
		\put(25,14) {\line(-1,-1){4}}
		\put(20,9) {\line(-1,-1){4}}
		\put(15,4) {\vector(-1,-1){4}}
		
		\put(20,6){\makebox(4,2){{ $\phi$}}}
		\end{picture}
		\caption{}
	\end{center}
\end{figure}
Set $J=\{a\in\mI(M)\mid f(x)\leq a,\forall x\in X\}$.
For all $a\in J$, the mapping $f$ maps $X$ into the interval $[0,a]$.
Since $([0,a];\oplus,\lam_a,0,a)$ is an
$MV$-algebra, there is a unique $MV$-homomorphism $\phi_a:F(X)\ra M$ such that
$(\phi_a\circ \tau)(x)=f(x)$ for all $x\in X$. Now, we claim that $\phi=\{\phi_a\mid a\in J\}$ is an $EMV$-morphism.

(i) Let $0=\min(F(X))$ and $1=\max(F(X))$. Clearly, $\{1\}$ is a full subset of $F(X)$.  For each $a\in J$,
$\phi_a(1)=a$. Clearly, $\{\phi_a(1)\mid a\in J\}$ is a full subset of $M$.

(ii) Let $a,b\in J$ such that $a\leq b$. Then $\phi_b \wedge a:F(X)\ra [0,a]$ defined by $(\phi_b\wedge a)(z)=(\phi_b(z)\wedge a)$, $z \in F(X)$, is an $MV$-homomorphism. Indeed,
for each $z,w\in F(X)$, by Proposition \ref{3.5}, we have
\[ (\phi_b \wedge a)(z')=\phi_b(z')\wedge a=\lam_b(\phi_b(z))\wedge a=\lam_a(\phi_b(z)),
\]
where $z'\in F(X)$ is the negation of the element $z\in F(X)$,
and
\begin{eqnarray*}
(\phi_b(z)\wedge a)\oplus (\phi_b(w)\wedge a)&=&\big((\phi_b(z)\wedge a)\oplus \phi_b(w)\big)\wedge
\big((\phi_b(z)\wedge a)\oplus a\big) \\
&=& (\phi_b(z)\oplus \phi_b(w))\wedge (a\oplus \phi_b(w))\wedge a\\
&=& (\phi_b(z)\oplus \phi_b(w))\wedge a.
\end{eqnarray*}
From $(\phi_b \wedge a)(\tau(x))=\phi_b(\tau(x))\wedge a=f(x)\wedge a=f(x)$ for all $x\in X$, it follows that
$\phi_b \wedge a=\phi_a$.

(iii) Let $a,b\in J$. Then there is $c\in J$ such that $a,b\leq c$. Clearly, $\phi_a(1),\phi_b(1)\leq \phi_c(1)$.

(iv) Since $\mathsf{e}(\phi)=\{1\}$, for each $x\in X$ we have $\tau(x)\in
Dom(\phi_a)$ for all $a\in J$. Let $a\in J$.
By definition of $\phi_a$, for each $x\in X$, (1) $f(x)\leq a$, (2)
$\phi_a(\tau(x))=f(x)$.
It follows that $\phi_a\circ \tau=f$ for all $a\in J$.

(i)--(iv) imply that $\phi$ is an $EMV$-morphism and the diagram given by
Figure 2 commutes.
Now, let $h=\{h_i\mid i\in K\}:F(X)\ra M$ be an $EMV$-morphism such that
$\mathsf{e}(h)=\{a_i\mid i\in K\}$, $h_i:[0,a_i]\to[0,h_i(a_i)]$, and
$(h_i\circ \tau)(x)=f(x)$ for all $x\in X$ and for all $i\in K$.
Since $\mathsf{e}(h)$ is a full subset of $F(X)$,
there exists $i\in K$ such that $1\leq a_i$ and so $1=a_i$. Set $K'=\{i\in
K\mid a_i=1\}$.

(v) It is easy to see that $h':=\{h_i\mid i\in K'\}:F(X)\ra M$ is an
$EMV$-morphism.

(vi) Let $i\in K$ and $j\in K'$. Since $h$ is an $EMV$-morphism, there is
$t\in K$ such that
$a_i,a_j\leq a_t$ and $h_i(a_i),h_j(a_j)\leq h_t(a_t)$. Hence, $t\in K'$ and
$h_i(x)=h_t(x)\wedge h_i(a_i)$ for all $x\in [0,a_i]$, which entails
$h\approx h'$. Thus, by Proposition \ref{9.3.1}, without loss of generality
we can assume that $\mathsf{e}(h)=\{1\}$. Now, we show that $h\approx \phi$. Put
$i\in K$. By definition,
there exists $b\in J$ such that $h_i(1)=h_i(a_i)\leq b$.  Since $h\circ
\tau=f$, then $h_i\circ \tau(x)=f(x)$ for all $x\in X$ and for all $i\in I$.
Moreover, the diagram given by Figure 3 commutes.
     \setlength{\unitlength}{1mm}
\begin{figure}[!ht]
    \begin{center}
        \begin{picture}(40,20)
        \put(12,17){\vector(2,0){12}}
        \put(8,14){\vector(0,-1){12}}
        \put(5,-2){\makebox(4,2){{ $[0,h_i(1)]$}}}
        \put(5,16){\makebox(4,2){{ $X$}}}
        \put(27,16){\makebox(4,2){{ $F(X)$}}}
        \put(15,19){\makebox(4,2){{ $\tau$}}}
        \put(1,7){\makebox(4,2){{$f$}}}
        \put(27,14){\vector(-1,-1){12}}
        \put(23,6){\makebox(4,2){{ $h_i$}}}
        \end{picture}
        \caption{ }
    \end{center}
\end{figure}

On the other hand, given $i\in K$, there is $b \in J$ such that $\phi_b:F(X)\ra [0,b]$ is an $MV$-homomorphism with
$\phi_b \circ \tau=f$ and $h_i(1)\leq b$. We can easily prove
that the map $\phi_b \wedge h_i(1):F(X)\ra [0,h_i(1)]$ sending $z\in F(X)$
to $\phi(z)\wedge h_i(1)$ is another $MV$-homomorphism
commuting the last diagram. Since $F(X)$ is the free $MV$-algebra on $X$,
then we have $h_i=\phi_b\wedge h_i(1)$, that is
$h_i(z)=\phi_b(z)\wedge h_i(1)$ for all $z\in F(X)$. Thus, $h\approx
\phi$. Therefore, $F(X)$ is the free $EMV$-algebra on $X$.
\end{proof}

We note that according to \cite[Thm 9.1.5]{mundici 1}, the free $MV$-algebra over $n$ generators is given by McNaughton functions $[0,1]^n\to [0,1]$, i.e. continuous functions $f:[0,1]^n\to [0,1]$ that are piece-vise linear with integer coefficients.

Let $M_1$ and $M_2$ be $EMV$-algebras, $X$ be a set and $\tau:X\ra M_1$ and
$f:X\ra M_2$ be two maps.
Consider an $EMV$-morphism $\varphi=\{\varphi_i\mid i\in I\}:M_1\ra M_2$. We know that
$\varphi\circ \tau$  is not defined in a usual way
(since $\varphi$ is not a map). Also, we need a new definition for equality
between $\varphi\circ \tau$ and $f$.
So, in Theorem \ref{9.7}, we used the following definition:
\[\varphi\circ \tau=f \Leftrightarrow (1)\ \ Im(\tau)\s Dom(\varphi_i),\ \forall i\in
I,\quad (2)\ \ \varphi_i\circ \tau(x)=f(x),\ \forall i\in I\ \forall x\in X.\]
We think that this definition is too strong. We can introduce a new version
based on properties of the components of $\varphi$ as follows:

In Proposition \ref{eq.property}, we have established some properties of
$EMV$-morphisms.
It was shown that there exists a tight  connection between the components
of an $EMV$-morphism $\varphi=\{\varphi_i\mid i\in I\}:M_1\ra M_2$.
So, we can use it to change the conditions under which $\varphi\circ \tau$ is ``similar" to $f$. That is, for each $x\in X$ and each $i\in I$ such that
$\tau(x)\leq a_i$, we have $f(x)\wedge \varphi_i(a_i)=\varphi_i(\tau(x))$.
(Indeed, for each $i\in I$, if $f(x)\leq \varphi_i(a_i)$, then
$f(x)=\varphi_i(\tau(x))$). In this case we write
$\varphi\circ \tau\sim f$.  Clearly, if $\varphi\circ \tau = f$, then $\varphi\circ \tau\sim f$.

Now, we use $\sim$ to introduce a new concept called a weakly free
$EMV$-algebra in the class of $EMV$-algebras when in Definition \ref{de:free} of the free $EMV$-algebra we change $f=\varphi\circ \tau$ by $f\sim \varphi\circ \tau$:

\begin{defn}\label{de:wfree}

Let $X\ne \emptyset$ be a set and $\tau:X\ra A$ be a map from $X$ to an $EMV$-algebra $A$. We say that an $EMV$-algebra $A$ is the {\it weakly free $EMV$-algebra} on $X$ if, for each $EMV$-algebra $M$ and for each map $f$ from $X$ into $M$, there is an $\approx$-unique $EMV$-morphism $\varphi =\{\varphi_i \mid i\in I\}:A\ra M$
with $\mathsf{e}(\varphi)=\{a_i\mid i\in I\}\subseteq \mathcal I(A)$ such that $\varphi\circ \tau\sim f$, which means that, for each $x\in X$ and each $i\in I$ such that
$\tau(x)\leq a_i$, we have $f(x)\wedge \varphi_i(a_i)=\varphi_i(\tau(x))$. The weakly free $EMV$-algebra on $X$ is denoted by $F_w(X)$.
\end{defn}

\begin{lem}\label{le:wfree}
Let $F_w(X)$ be a weakly free $EMV$-algebra on $X$ with a mapping $\tau$. Then $\tau$ is an injective mapping.
\end{lem}

\begin{proof}
Suppose the converse, i.e. there are two different points $x_1,x_2\in X$ such that $\tau(x_1)=\tau(x_2)$. Let $M$ be an $EMV$-algebra having at least two elements and let $f:X\to M$ be a mapping such that $f(x_1)\ne f(x_2)$. There is an $EMV$-morphism $\phi=\{\phi_i\mid i \in I\}:F_w(X) \to M$ with $\mathsf e(\phi)=\{a_i\mid i \in I\}$ such that if $x\in X$ with $\tau(x)\le a_i$, then $f(x)\wedge \phi_i(a_i)=\phi_i(\tau(x))$.

Since $\{a_i\mid i \in I\}$ is full in $F_w(X)$, there is $i\in I$ such that $\tau(x_1)=\tau(x_2)\le a_i$. Similarly, since $\{\phi_i(a_i)\mid i \in I\}$ is full in $M$, there is $j\in I$ such that $f(x_1),f(x_2)\le \phi_j(a_j)$. Finally, by (iv) of Definition \ref{9.1}, there is $k\in I$ such that $a_i,a_j\le a_k$ and $\phi_i(a_i),\phi_j(a_j)\le \phi_k(a_k)$. Hence,
$f(x_1)=f(x_1)\wedge \phi_k(a_k)=\phi_k(\tau(x_1))= \phi_k(\tau(x_2))= f(x_2)\wedge \phi_k(a_k)=f(x_2)$ which is a contradiction.
\end{proof}

 \begin{thm}\label{weak free}
 Let $F(X)$ be a free MV-algebra on a set $X\ne \emptyset$. Then $F(X)$ is a weakly free $EMV$-algebra on $X$.
 \end{thm}

\begin{proof}
Let $M$ be an $EMV$-algebra, $f:X\to M$ be any mapping,  $\tau:X \to F(X)$ be an embedding, and see the diagram given by Figure 2. By the basic representation theorem of $EMV$-algebras,
$M$ is an $MV$-algebra or there exists an $MV$-algebra $N$ such that $M$ is a maximal ideal of $N$.
	 If $M$ is an $MV$-algebra, the proof is evident from \cite{mundici 1}.

If $M$ is not an $MV$-algebra, there is an
	 $MV$-homomorphism $\phi:F(X)\ra N$ such that $\phi\circ \tau(x)=f(x)$ for all $x\in X$.
	 Clearly, $\{1\}$, the greatest element of $F(X)$, is a full subset of $F(X)$. For each $a\in \mI(M)$, define
	 $\phi_a:[0,1]\ra [0,a]$ by $\phi_a(z)=\phi(z)\wedge a$ for all $z\in F(X)$.
	 Since $M$ is an ideal of $N$ and $a\in M$, then $[0,a]\s M$.
	 It can be easily seen that $\phi_a$ is an $MV$-homomorphism, for all $a\in \mI(M)$,
	 hence, $\beta=\{\phi_a\mid a\in \mI(M)\}$ is an $EMV$-morphism from $F(X)$ to $M$.
	 Let $x\in X$ and $a\in\mI(M)$ (clearly, $\tau(x)\in Dom(\phi_a)$). Then by definition,
	 $\phi_a(\tau(x))=\phi(\tau(x))\wedge a=f(x)\wedge a$. That is, $\beta\circ \tau\sim f$.
	
Now, we show that $\beta$ is $\approx$-unique.
	 Let $\gamma=\{\gamma_j\mid j\in I\}:F(X)\ra M$ be an $EMV$-morphism such that $\mathsf{e}(\gamma)=\{c_j\mid j\in I\}$
	 and $\gamma\circ \tau\sim f$. Clearly, $1\in \mathsf{e}(\gamma)$ (since $\mathsf{e}(\gamma)$ is a full subset of $F(X)$).
	 Let $K=\{j\in I\mid c_j=1\}$ and $h:=\{\gamma_j\mid j\in K\}$.
	 Then by Proposition \ref{eq.property}(iii), $h:F(X)\ra M$ is an $EMV$-morphism which is
	 similar to $\gamma$. We claim that $h\approx \beta$.  That is, for each $j\in K$, there is $a\in\mI(M)$ such that
	 $\gamma_j(c_j)=\gamma_j(1)\leq \phi_a(1)=a$ and for each $x\in [0,1]=F(X)$, we have $\gamma_j(x)=\phi_a(x)\wedge \gamma_j(1)$
	 (we note that $Dom(\gamma_j)=Dom(\phi_a)=F(X)$).
	 Put $j\in K$. Consider the map $f_j:X\ra M$ defined by $f_j(x)=f(x)\wedge \gamma_j(1)$. Set $a=\gamma_j(1)$. Then
	 $a\in \mI(M)$. Since $\gamma\circ \tau\sim f$, then for each $x\in X$, $\gamma_j(\tau(x))=f(x)\wedge \gamma_j(1)=f(x)\wedge a=\phi_a(\tau(x))$
	 and so we can say that the following diagram for $MV$-algebras given by Figure 4 is ``$\sim$-commutative",
	 \setlength{\unitlength}{1mm}
	 \begin{figure}[!ht]
	 	\begin{center}
	 		\begin{picture}(40,20)
	 		\put(12,17){\vector(2,0){12}} \put(8,14){\vector(0,-1){12}}
	 		\put(5,-2){\makebox(4,2){{ $[0,a]$}}}
	 		\put(5,16){\makebox(4,2){{ $X$}}}
	 		\put(27,16){\makebox(4,2){{ $F(X)$}}}
	 		\put(15,19){\makebox(4,2){{ $\tau$}}}
	 		\put(1,7){\makebox(4,2){{$f_j$}}}
	 		\put(25,14){\vector(-1,-1){12}}
	 		\put(27,12){\vector(-1,-1){12}}
	 		\put(15,10){\makebox(4,2){{ $\gamma_j$}}}
	 		\put(23,6){\makebox(4,2){{ $\phi_a$}}}
	 		\end{picture}
	 		\caption{ }
	 	\end{center}
	 \end{figure}
which implies that $\phi_a=\gamma_j$ ($F(X)$ is the free $MV$-algebra on $X$). That is, for each $x\in F(X)$,
$\gamma_j(x)=\gamma_j(x)\wedge \gamma_j(1)=\phi_a(x)\wedge \gamma_j(1)=\phi_a(x)\wedge a$. It follows that
$h\approx\beta$ and so $\gamma\approx h\approx\beta$. Therefore, $F(X)$ is a weakly free $EMV$-algebra on $X$.
\end{proof}

We note the elements of the free $MV$-algebra on an infinite set $X$ are McNaughton functions $[0,1]^\kappa \to [0,1]$, where $\kappa =|X|$, see \cite{mundici 1}.

\begin{thm}\label{weakfree prop}
If $M_1$ and $M_2$ are weakly free $EMV$-algebras on a set $X$, then they are $\approx$-isomorphic.
\end{thm}

\begin{proof}
	Let $X$ be a set and $M_1$ and $M_2$ be weakly free $EMV$-algebras on $X$ with the maps $\tau_1:X\ra M_1$ and $\tau_2:X\ra M_2$.
	Then there exist two $\approx$-unique $EMV$-morphisms $\varphi:M_1\ra M_2$ and $\psi:M_2\ra M_1$ such that $\varphi\circ \tau_1\sim\tau_2$ and
	$\psi\circ \tau_2\sim\tau_1$. Let $\mathsf{e}(\varphi)=\{a_i\mid i\in I\}$ and $\mathsf{e}(\psi)=\{b_j\mid j\in J\}$.
	Consider the $EMV$-morphisms $\psi\circ \varphi:M_1\ra M_1$ and $\varphi\circ \psi:M_2\ra M_2$.
	We show that $\psi\circ \varphi\approx Id_{M_1}$ and $\varphi\circ \psi\approx Id_{M_2}$ (see Proposition \ref{equal to identity} and
	Corollary \ref{isomorphism}) which implies that $M_1$ and $M_2$ are $\approx$-isomorphic $EMV$-algebras.
	Set $U=\{(i,j)\in I\times J\mid \varphi(a_i)\leq b_j\}$. We know that
	$\psi\circ\varphi=\{\psi_j\circ\varphi_i\mid (i,j)\in U\}$ and $\mathsf{e}(\psi\circ\varphi)=\{a_i\mid (i,j)\in U\}$.
	For simplicity, set $\beta=\psi\circ\varphi$.
	First, we must show that for each $x\in X$ and each $(i,j)\in U$ such that $\tau_1(x)\leq a_i$, then
	$\psi_j\circ\varphi_i(\tau_1(x))=\tau_1(x)\wedge (\psi_j\circ \varphi_i(a_i))$.  Put $x\in X$ and $i\in I$ such that $\tau_1(x)\leq a_i$.
	From $\varphi\circ \tau_1\sim\tau_2$ it follows that $\varphi_i(\tau_1(x))=\tau_2(x)\wedge\varphi_i(a_i)$ and so
	$\psi_j(\varphi_i(\tau_1(x)))=\psi_j(\tau_2(x)\wedge\varphi_i(a_i))$. 	
	Since $\psi$ is an $EMV$-morphism, there exists $k\in J$ such that 	 $b_k\geq b_j,\tau_2(x)$ and $\psi_k(b_k)\geq\psi_j(b_j)$,
	which imply that $\psi_j(\varphi_i(\tau_1(x)))=\psi_j(\tau_2(x)\wedge\varphi_i(a_i)) =\psi_k(\tau_2(x)\wedge\varphi_i(a_i))\wedge\psi_j(b_j)
	=\psi_k(\tau_2(x))\wedge \psi_k(\varphi_i(a_i))\wedge \psi_j(b_j)$. Since $\varphi_i(a_i)\leq b_j$, then
	$\psi_k(\varphi_i(a_i))\wedge \psi_j(b_j)=\psi_j(\varphi_i(a_i))$. Also,
	$\psi_k(b_k)\geq \psi_k(b_j)\geq\psi_j(b_j)\geq\psi_j(\varphi_i(a_i))$, so that
	\begin{eqnarray*}
	\psi_j\circ\varphi_i(\tau_1(x))&=&\psi_k(\tau_2(x))\wedge \psi_j(\varphi_i(a_i))=\tau_1(x)\wedge \psi_k(b_k)\wedge\psi_j(\varphi_i(a_i)),\mbox{ since $\psi\circ \tau_2\sim\tau_1$}\\
	&=&  \tau_1(x)\wedge\psi_j(\varphi_i(a_i))=\tau_1(x)\wedge(\psi_j\circ \varphi_i(a_i)).
	\end{eqnarray*}

That is, the diagram given by Figure 5 is ``$\sim$-commutative".
\setlength{\unitlength}{1mm}
\begin{figure}[!ht]
	\begin{center}
		\begin{picture}(40,20)
		\put(12,17){\vector(2,0){12}} \put(8,14){\vector(0,-1){12}}
		\put(5,-2){\makebox(4,2){{ $M_1$}}}
		\put(5,16){\makebox(4,2){{ $X$}}}
		\put(27,16){\makebox(4,2){{ $M_1$}}}
		\put(15,19){\makebox(4,2){{ $\tau_1$}}}
		\put(1,7){\makebox(4,2){{$\tau_1$}}}
		\put(25,14){\vector(-1,-1){12}}
		\put(25,6){\makebox(4,2){{ $\psi\circ\varphi$}}}
		\end{picture}
		\caption{ }
	\end{center}
\end{figure}
Clearly, $Id_{M_1}\circ\tau_1\sim\tau_1$. Thus $Id_{M_1}\approx\psi\circ\varphi$ (since $M_1$ is a weakly free $EMV$-algebra on $X$).
In a similar way, we can show that
$Id_{M_2}\approx \varphi\circ\psi$. Therefore, $M_1$ and $M_2$ are $\approx$-isomorphic $EMV$-algebras.
\end{proof}

\begin{lem}\label{le:X}
If $|X|=|X'|$, then $F_w(X)$ and $F_w(X')$ are $\approx$-isomorphic.
\end{lem}

\begin{proof}
Let $\tau:X \to F_w(X)$ and $\tau':X'\to F_w(X')$ be injective mappings determining $F_w(X)$ and $F_w(X')$, respectively. Since $|X|=|X'|$, there is a bijective mapping $\beta:X\to X'$.
There are an $\approx$-unique $EMV$-morphism $\phi=\{\phi_i\mid i \in I\}:F_w(X)\to F_w(X')$ with $\mathsf e(\phi)=\{a_i\mid i\in I\}$ and an $\approx$-unique $EMV$-morphism $\psi=\{\psi_j\mid j \in J\}:F_w(X')\to F_w(X)$ with $\mathsf e(\psi)=\{b_j\mid j \in J\}$ such that  $\tau'\circ \beta\sim \phi \circ \tau$ and $\tau\circ \beta^{-1} \sim \psi \circ \tau'$.
We claim to prove that $\psi\circ \phi \approx Id_{F_w(X)}$. To show that, we establish that $\psi\circ \phi \circ \tau \sim \tau$. Put $U=\{(i,j)\in I\times J\mid \phi_i(a_i)\le b_j\}$. Then $\psi\circ \phi=\{\psi_j\circ \phi_i\mid (i,j)\in U\}$ and $\mathsf e(\psi\circ \phi)=\{a_i\mid (i,j)\in U\}$.

Due to $\tau'\circ \beta\sim \phi \circ \tau$ and $\tau\circ \beta^{-1} \sim \psi \circ \tau'$, we have
\begin{eqnarray*}
\phi_i(\tau(x)) &=& \tau'(\beta(x))\wedge \phi_i(a_i), \quad \mbox{if }\ x\in X, \tau(x)\le a_i,\\
\psi_j(\tau'(x'))&=& \tau(\beta^{-1}(x'))\wedge \psi_j(b_j),\quad \mbox{if }\ x'\in X', \tau'(x')\le b_j.
\end{eqnarray*}
Then $\tau'(\beta(x))\wedge \phi_i(a_i)=\phi_i(\tau(x))$ if $x\in X$ and $\tau(x)\le a_i$, and $\tau(\beta^{-1}(x'))\wedge \psi_j(b_j)= \psi_j(\tau'(x'))$ if $x'\in X$ and $\tau'(x')\le b_j$. Take $(i,j)\in U$ and $\tau(x)\le a_i$ for $x \in X$. Whence
\begin{eqnarray*}
\psi_j(\phi_i(\tau(x)))&=& \psi_j(\tau'(\beta(x)))\wedge (\psi_j\circ \phi_i(a_i))\\
&=& \tau(\beta^{-1}(\beta(x)))\wedge \psi_j(b_j)\wedge (\psi_j\circ \phi_i(a_i))\\
&=& \tau(x) \wedge (\psi_j\circ \phi_i(a_i)).
\end{eqnarray*}
Then $(\psi \circ \phi)\circ \tau \sim\tau$ which yields $\psi\circ \phi \approx Id_{F_w(X)}$.

In the same way we prove $\phi \circ \psi \approx Id_{F_w(X')}$. Therefore, $F_w(X)$ and $F_w(X')$ are $\approx$-isomorphic.
\end{proof}

\begin{lem}\label{le:F_w(X)}
    Let $F_w(X)$ be a weakly free $EMV$-algebra on a set $X$ with an embedding $\tau$. If $A$ is a full
subalgebra of $F_w(X)$ containing $\tau(X)$, then $A$ is equal to $F_w(X)$.
\end{lem}

\begin{proof}
We follow ideas and notations from the proof of Lemma \ref{le:F(X)}.

Let $\tau: X\to F_w(X)$ and $\tau': X \to A$ be such mappings that $\tau(x)=\tau'(x)$ for each $x \in X$. Let $\mu:A \to F_w(X)$ be the natural embedding, i.e. $\mu(y)=y$ for each $y \in A$.
Since $A$ is a full subalgebra of $F_w(X)$, the inclusion map $\mu:A\ra
F_w(X)$ can be viewed as an $EMV$-morphism $\mu =\{\mu_a\mid a \in \mathcal I(A)\}$, where $\mu_a=\mu|_{[0,a]}$ for each $a\in \mathcal I(A)$.
The definition of $F_w(X)$ entails that there is an $\approx$-unique $EMV$-morphism $\phi=\{\phi_i\mid i \in I\}: F_w(X)\to A$ with $\mathsf e(\phi)=\{a_i \mid i \in I\}$ such that $\phi \circ \tau \sim\tau'$. That is, if $x \in X$ and $\tau(x)\le a_i$, then
$$
\tau'(x)\wedge \phi_i(a_i)=\phi_i(\tau(x)).
$$

Set $U:=\{(i,a)\in I\times \mI(A)\mid \phi_i(a_i)\leq a\}$. By
definition, $\mu\circ\phi=\{\mu_a\circ \phi_i\mid (i,a)\in U\}$ with $\mathsf e(\mu\circ \phi)=\{a_i\mid (i,a)\in U\}$.

We assert $(\mu\circ \phi)\circ \tau \sim \tau$. Check, for $\tau(x) \le a_i\le a$:
$$
\tau(x)\wedge (\mu_a\circ \phi_i)(a_i) =\tau(x)\wedge \phi_i(a_i)=\tau'(x)\wedge \phi_i(a_i)= \phi_i(\tau(x))=(\mu_a\circ \phi_i)\circ \tau(x).
$$
Then clearly $\mu\circ \phi \approx Id_{F_w(X)}$.

Let $v$ be an arbitrary element of $F_w(X)$. Then there exists $i\in
I$ such that $v\leq a_i$ and $v\leq \phi_i(a_i)$.  Since $\mu\circ\phi\approx Id_{F_w(X)}$, by Proposition \ref{equal to identity}, for each $a\in\mI(A)$ with $\phi_i(a_i)\leq a$, we have $\mu_a\circ \phi_i(v)=v$. That is, $v=\phi_i(v)\in A$. Therefore, $A=F_w(X)$.
\end{proof}

\begin{cor}\label{co:X finite}
If $X$ is a finite set, then the free $EMV$-algebra on $X$ and the weakly free $EMV$-algebra on $X$ are $\approx$-isomorphic.
\end{cor}

\begin{proof}
Let $X$ be a finite non-empty set. By Theorem \ref{9.7} and Theorem \ref{weak free}, if $F(X)$ is a a free $MV$-algebra on $X$, then it is both a free $EMV$-algebra and a weakly free $EMV$-algebra. Since every free MV-algebra is a weakly free $EMV$-algebra, Theorem \ref{weakfree prop} gives the result.
\end{proof}

\section{Conclusion}

Recently the authors introduced in \cite{Dvz} a new class of algebras called $EMV$-algebras. This structure has a bottom element but not necessarily a top element. They have features close to $MV$-algebras because for every idempotent $a$, the interval $[0,a]$ forms an $MV$-algebra. The family $\mathcal{EMV}$ of $EMV$-algebras is a variety with respect to $EMV$-homomorphisms. However, to study some important objects like free $EMV$-algebras, instead of an $EMV$-homomorphism from $M_1$ into $M_2$, we need an $EMV$-morphism $M_1\to M_2$ which is a family of $MV$-morphisms defined on intervals $[0,a]$, see Definition \ref{9.1}. In the paper, we studied their basic properties of $EMV$-morphisms as e.g. a composition of two $EMV$-morphisms and an equivalence, $\approx$, called similarity, between $EMV$-morphisms from $M_1$ into $M_2$. We note that the composition is not associative, it is only $\approx$-associative, that is associative up to $\approx$, Proposition \ref{pr:assoc}.
$EMV$-morphisms were applied to introduce three categories of $EMV$-algebras. We have showed that the category with objects $EMV$-algebras and morphisms connected with classes of standard $EMV$-morphisms is equivalent to the category where morphisms are classes corresponding to strong $EMV$-homomorphisms, Theorem \ref{cat}. We have studied free $EMV$-algebras on a set $X$ applying $EMV$-morphisms. We have established that if $X$ is a finite non-empty set, then the free $MV$-algebra on $X$ is also a free $EMV$-algebra on $X$, see Theorem \ref{9.7}. To show an analogous result for infinite $X$, we have introduced the so-called weakly free $EMV$-algebra on $X$. In Theorem \ref{weak free}, we have proved that every free $MV$-algebra on $X$ is also a weakly free $EMV$-algebra on $X$.

The suggested $EMV$-morphisms seem to be useful to establish many important properties of the class of $EMV$-algebras also in future.

\section{New}

\begin{prop}
Let $(M;+,0)$ be a commutative monoid with neutral element $0$ satisfies
the following conditions:
\begin{itemize}[nolistsep]
    \item[{\rm (i)}] $(M,\leq)$ is a poset with the least element $0$;
    \item[{\rm (ii)}]  for each $x,y\in M$, there is $a\in \mI(M)$ such
that $x,y\leq a$;
    \item[{\rm (iii)}]  for each $b\in \mI(M)$, the element
    $\lambda_{b}(x)=\min\{z\in[0,b]\mid x\oplus z=b\}$
    exists in $M$ for all $x\in [0,b]$, and the algebra
$([0,b];+,\lambda_{b},0,b)$ is an $MV$-algebra;
    \item[{\rm (iv)}] for all $b\in \mI(M)$, and  for all $x,y\in [0,b]$,
we have $x\leq y$ \iff $\lam_b(x)+y=b$.
\end{itemize}
Then $M$ is an $EMV$-algebra.
\end{prop}

\begin{proof}
First we show that $(M,+,0)$ is a commutative ordered monoid with respect
to $\leq$. Let $x,y,z\in M$ such that $x\leq y$.
By (ii), there exists $b\in \mI(M)$ such that $x,y,z\leq b$. Since
$([0,b];+,\lambda_{b},0,b)$, by (iv), $x+z\leq y+z$ and so $(M,+,0)$ is a
commutative ordered monoid with respect to $\leq$.
Suppose that $x\sqcup y$ be the least upper bound of $x$ and $y$ in
$([0,b];+,\lambda_{b},0,b)$.
We claim that $x\vee_{_b} y$ is the least upper bound of $x$ and $y$ in
$(M,\leq)$.
Let $u\in M$ be an upper bound of $\{x,y\}$. Then by (ii), there is
$a\in\mI(M)$ such that $b,u\leq a$.
Let $b\wedge_{_a} u$ be the greatest lower bound of $\{b,u\}$ in the
$MV$-algebra $([0,a];+,\lambda_{a},0,a)$.
Clearly, $b\wedge_{_a} u\in [0,b]$ and is an upper bound for $\{x,y\}$
(note that the by (iv), the partially order relation on the $MV$-algebras
$[0,b]$ and $[0,a]$ are coincide).
It follows that $x\vee_{_b} y\leq b\wedge_{_a} u\leq u$ and so $x\vee_{_b}
y$ is the least upper bound of $x$ and $y$ in $(M,\leq)$.
In a similar way, $\{x,y\}$ has the greatest lower bound in $(M,\leq)$ and
so $(M,\leq)$ is a lattice with the least element $0$.
Since each $MV$-algebra $([0,b];+,\lambda_{b},0,b)$ is a distributive
lattice, we can easily obtain that
$(M,\leq)$ is a distributive lattice, too. Summing up the above results,
$M$ is an $EMV$-algebra.
\end{proof}



\begin{thebibliography}{AnCo}
\footnotesize{


\bibitem[AnFe]{Anderson1}
M. Anderson, T. Feil, Lattice-Ordered Groups: An Introduction,
{\it Springer Science and Business Media}, USA, 1988.

\bibitem[Bly]{Bly}
T.S. Blyth, Lattice and Ordered Algebraic Structures, {\it Springer-Verlag}, London,  2005.

\bibitem[BuSa]{burris}
S. Burris, H.P. Sankappanavar, A Course in Universal Algebra, {\it Springer-Verlag}, New York, 1981.

\bibitem[Cha]{Cha}
C.C. Chang,  Algebraic analysis of many valued logics, {\it Transactions of the American Mathematical Society} {\bf 88} (1958),  467--490.

\bibitem[CDM]{mundici 1}
R. Cignoli, I.M.L. D'Ottaviano and D. Mundici, Algebraic Foundations of Many-Valued Reasoning,
{\it Springer Science and Business Media}, Dordrecht, 2000.

\bibitem[CoDa]{CoDa}
P. Conrad, M.R. Darnel, Generalized Boolean algebras in lattice-ordered groups, {\it Order} {\bf 14} (1998), 295--319.

\bibitem[DiLe]{DiLe}
A. Di Nola, A. Lettieri, Equational characterization of all varieties of MV-algebras, {\it Journal of Algebra} {\bf 221} (1999), 463-–474.

\bibitem[DiRu]{DiRu}
A. Di Nola, C. Russo, The semiring-theoretic approach to MV-algebras:
A survey, {\it Fuzzy Sets Syst.} {\bf 281} (2015), 134--154.

\bibitem[Dvu2]{Dvu2}
A. Dvure\v{c}enskij, Pseudo MV-algebras are intervals in $\ell$-groups,
{\it Journal of the Australian Mathematical Society}, {\bf 72} (2002), 427--445.

\bibitem[DvPu]{DvPu}
A. Dvure\v censkij, S. Pulmannov\'a,   New
	Trends in Quantum Structures, {\it Kluwer Academic Publ.},
Dordrecht, {\it Ister Science}, Bratislava, 2000.

\bibitem[DvZa]{Dvz}
A. Dvure\v censkij, O. Zahiri, On $EMV$-algebras, http://arxiv.org/abs/1706.00571

\bibitem[DvZa1]{Dvz1}
A. Dvure\v censkij, O. Zahiri, The Loomis--Sikorski Theorem for $EMV$-algebras, {\it Journal of the Australian Mathematical Society}, to appear, http://arxiv.org/abs/1707.00270

\bibitem[DvZa2]{Dvz3}
A. Dvure\v censkij, O. Zahiri, States on $EMV$-algebras, http://arxiv.org/abs/1708.06091


\bibitem[GaTs]{Tsinakis}
N. Galatos, C. Tsinakis, Generalized MV-algebras, {\it Journal of Algebra} {\bf 283} (2005), 254--291.

\bibitem[GeIo]{georgescu}
G. Georgescu and A. Iorgulescu, Pseudo $MV$-algebras, {\it Multiple-Valued Logics} {\bf 6} (2001), 193--215.

\bibitem[GlHo]{Glass}
A.M.W. Glass, W. Holland, Lattice-Ordered Groups: Advances and Techniques,  {\bf 48}, Kluwer Academic Publishers, Dordrecht, 1989.

\bibitem[Kom]{Kom}
Y. Komori, Super \L ukasiewicz propositional logics, {\it Nagoya Mathematical Journal} {\bf 84} (1981), 119--133.

\bibitem[LuZa]{LuZa}
W.A.J. Luxemburg, A.C. Zaanen, Riesz Spaces, Vol 1, {\it North-Holland Publishers Co.}, Amsterdam, London 1971.

\bibitem[MaL]{MaL}
S. Mac Lane, Categories for the Working Mathematician, {\it Springer-Verlag},  New York, Heidelberg, Berlin, 1971.

\bibitem[Mun1]{Mun}
D. Mundici,  Interpretation
of AF $C^*$-algebras in \L ukasiewicz sentential calculus, {\it Journal of Functional Analysis} {\bf 65} (1986),  15--63.

\bibitem[Mun2]{Mun2}
D. Mundici, Averaging the truth-value in \L ukasiewicz logic, {\it Studia Logica},
 {\bf 55} (1995), 113--127.

\bibitem[Mun3]{mundici 2}
D. Mundici, Advanced {\L}ukasiewicz calculus and MV-algebras, {\it Springer}, Dordrecht, Heidelberg, London, New York, 2011.

\bibitem[Rac]{Rach}
J. Rach\r{u}nek, A non-commutative generalization of MV-algebras, {\it Czechoslovak Mathematical Journal} {\bf 52} (2002), 255--273.

\bibitem[ShLu]{Shang}
Y. Shang, R. Lu,
Semirings and pseudo MV algebras, {\it Soft Computing} {\bf 11} (2007), 847--853.

 \bibitem[Sto1]{Sto1}
 M.H. Stone, Applications of the theory of Boolean rings to general topology,
 {\it Transactions of the American Mathematical Society} {\bf 41} (1937), 375--481.

 \bibitem[Sto2]{Sto2}
 M.H. Stone, Topological representation of distributive lattices and Brouwerian logics, {\it \v{C}asopis pro p\v{e}stov\'an\'\i\ matematiky a  fysiky} {\bf 67} (1938), 1--25.

}


\end{thebibliography}
\end{document}